\documentclass{amsart}

\usepackage{amssymb,amsmath}
\usepackage{url}
\usepackage{amscd}
\usepackage{texdraw}
\usepackage[all]{xy}
\usepackage{fancyhdr}
\usepackage{amsmath}
\usepackage{latexsym}
\usepackage{graphicx}
\usepackage{subcaption}
\DeclareMathOperator{\sech}{sech}
\usepackage{amscd,color}
\usepackage[shortlabels]{enumitem}

\theoremstyle{plain}
\newtheorem{thm}{Theorem}[section]

\newtheorem*{thmA}{Theorem A}
\newtheorem*{thmB}{Theorem B}
\newtheorem*{thmC}{Theorem C}
\newtheorem*{thmD}{Theorem D}
\newtheorem*{thmE}{Theorem E}
\newtheorem*{thmF}{Theorem F}

\newtheorem{cor}[thm]{Corollary}
\newtheorem{prop}[thm]{Proposition}
\newtheorem{lemma}[thm]{Lemma}
\newtheorem{defn}{Definition}

\newtheorem{remark}{Remark}[section]

\input txdtools

\newcommand{\calb}{{\mathcal B}}
\newcommand{\calc}{{\mathcal C}}

\newcommand{\calf}{{\mathcal F}}

\newcommand{\calo}{{\mathcal O}}
\newcommand{\calp}{{\mathcal P}}

\newcommand{\calr}{{\mathcal R}}
\newcommand{\cals}{{\mathcal S}}

\newcommand{\calz}{{\mathcal Z}}

\newcommand{\CC}{{\mathbb C}}
\newcommand{\DD}{{\mathbb D}}

\newcommand{\NN}{{\mathbb N}}
\newcommand{\QQ}{{\mathbb Q}}

\newcommand{\RR}{{\mathbb R}}

\renewcommand{\hat}{\widehat}

\newcommand{\la}{\lambda}

\begin{document}
\title{ Dynamics of the meromorphic families $f_{\la}=\la \tan^p z^q$}
 \footnote{2000 Mathematics Subject Classification. 37F30, 37F20, 37F10,30F30, 30D30, 30A68.}

\author{Tao Chen  and Linda Keen}

\address{}
\email{}

\thanks{}

\subjclass[2000]{58F23, 30D05}

 \begin{abstract}  This paper continues our investigation of the dynamics of families of transcendental meromorphic functions with finitely many singular values all of which are finite.   Here we  look at a generalization of the family of polynomials $P_a(z)=z^{d-1}(z- \frac{da}{(d-1)})$, the family $f_{\la}=\la \tan^p z^q$.  These functions have a super-attractive fixed point, and, depending on $p$, one or two asymptotic values.   Although many of the dynamical properties generalize, the existence of an essential singularity and of poles of multiplicity greater than one implies that significantly different techniques are required here.   Adding transcendental methods to standard ones, we give a description of the dynamical properties; in particular we prove the Julia set of a hyperbolic map is either connected and locally connected or a Cantor set.   We also give a description of the parameter plane of the family $f_{\la}$.  Again there are similarities to and differences from  the parameter plane of the family $P_a$ and again  there are new techniques.   In particular, we prove there is dense set of points on the boundaries of the hyperbolic components that are accessible along curves and we characterize these  points.
 
 \end{abstract}

\maketitle
\section{Introduction}
\label{intro}

This paper is part of an ongoing  program to understand the dynamic and parameter spaces of dynamical systems generated by transcendental meromorphic functions with a finite number of singular values.  In holomorphic dynamics in general, the singular values control the periodic and eventually periodic stable behavior.  Functions with finitely many singular values have no non-periodic stable behavior, so the singular values provide a natural way to parameterize such a family in order to visualize how the dynamics vary across it.   In such full generality, though, that is not really practical.   What does work is to constrain the dynamic behavior of most of the singular values, for example by insisting that their orbits tend to an attracting fixed point with constant multiplier, and to study the dynamic behavior as the others vary freely.   One virtue of this approach is that it allows one to use computer graphics to gain some intuition.       

This philosophy has driven the work on rational dynamics, most famously the study of the quadratic family $P_c(z)=z^2+c$.   Computing the orbits of the critical value $c$ leads one to the well-studied picture of the Mandelbrot set,  which shows that, generically, the critical value either escapes to infinity or tends to an attractive periodic cycle.   Many of the techniques initiated by Douady, Hubbard, Sullivan and a lot of others were developed to understand it.  (See, e.g. \cite{M2}).   Our work on transcendental functions extends the theory for rational functions, so let us begin with a brief review of that story, first for quadratics and then for polynomials.  See section~\ref{basics} for definitions.

In the quadratic case, what we are calling the generic values for $c$ are those for which the dynamics are hyperbolic.  This set is partitioned into infinitely many open, bounded, simply connected sets (Mandelbrot-like hyperbolic components)  and one unbounded set homeomorphic to the outside of a disk.  These components are separated by a ``bifurcation locus'', where the dynamics undergoes changes.    The dynamics for $c$'s in the  bounded sets is quite different from those in  the unbounded set.  In the former, the unstable (or Julia) set is always connected and locally connected, while in the latter it is always a Cantor set, and the action of $P_c$  on this Cantor set is conjugate to the shift map on two symbols.  Thus, this unbounded component is called a {\em shift locus}.   

This situation can be generalized to polynomials of degree $d>1$.   A ``dynamically natural'' one dimensional slice can be obtained by constraining the origin be a critical point of multiplicity $d-1$.  Thus, all but one of the singular values is fixed, and the origin is always an attracting fixed point.  There is one more ``free'' critical value $a$, and one can ask how the dynamics depend on it.  Such polynomials can be written in the form
$$P_a(z)=z^{d-1}(z- \frac{da}{(d-1)})$$
They were studied in \cite{Roe, EMS}) where it is shown that, in addition to Mandelbrot hyperbolic components and a shift locus, there is a new type of bounded hyperbolic component called a {\em capture component}.  In it, the critical point $0$ attracts the free critical point $a$.  We use ``shell component" to refer to the Mandlelbrot-like components, to distinguish them from the capture and shift-locus components.

It is natural to ask what different, if anything, happens for transcendental functions.  There are three essential differences between polynomial and transcendental maps. 

\begin{itemize}
 \item  For a polynomial, infinity is always a fixed critical value with no preimages, but for a transcendental map, it is an essential singularity that has preimages, and those preimages are essential singularities of iterates of the map.
 
\item Transcendental maps are infinite to one.
 
 \item  In addition to critical values, transcendental maps can have  another type of singular value where the local inverses are not all well defined.  These are called ``asymptotic values''.    The standard example is $0$ for $e^z$. 
 \end{itemize}
 
If one restricts to families with finitely many singular values, and in addition require that all the singular values are finite, there are many parallels with polynomials.  For example, in the tangent family $\la \tan z$, the functions have two symmetric asymptotic values, $\pm \la i$, and these are the only singular values.  Because of the symmetry, there are symmetries in the dynamics and in the parameter plane.   In \cite{KK1}, we studied this family and showed that the parameter space has a structure analogous to that for the quadratic maps.  There are shell components that are hyperbolic components   much like the hyperbolic components of the Mandelbrot set, where the asymptotic values tend to attracting periodic orbits.  There is also a shift locus component where both asymptotic values are attracted to the fixed point $0$. On this component, the map is conjugate on its Julia set  to a shift map not on two symbols, but on infinitely many symbols.   
 
A natural generalization of the tangent family is where there are two asymptotic values, no critical values, and, rather than require the asymptotic values to be symmetric, one requires that $0$ be an attractive fixed point with its multiplier fixed across the family. In this situation, at least one asymptotic value is always attracted to zero.  We proved in \cite{CJK1, CJK2} that the structure of the parameter plane was analogous to that for \emph{rational} maps of degree two that have a fixed point with constant multiplier: there are two collections of shell components and the complement of their closure is an annular  shift locus.   

The next step in our program is to look at families with more than two singular values and study one dimensional slices of their parameter spaces to see what new structures we find.   In this paper, we consider functions that are transcendental analogues of the polynomials $P_a$: the family $f_{\la}(z)=\la \tan^p z^q$ for $pq>1$. These functions have a critical value of multiplicity $pq$ at the origin and either one or a pair of symmetric asymptotic values.   The origin is a super-attractive fixed point and the basin of points attracted to the origin, $B$, is never empty.   The component of $B$ containing the origin is the immediate basin $B_0$.   

Our first set of results is about the dynamic plane.   We show that, like quadratic polynomials and tangent maps, there is a dichotomy of the connectivity of the Julia set for hyperbolic maps.  The techniques are similar to those in  \cite{CJK1, CJK2}  but need to be modified because there are poles of multiplicity greater than one in the Julia set.  To deal with this we use an estimate of Rippon and Stallard, \cite{RS}.  

\begin{thmA} The immediate basin of the origin $B_0$ is completely invariant if, and only if, it contains the asymptotic value(s).   If it does, it is infinitely connected, the Julia set is a Cantor set, and $f_{\la}$ on the Julia set is topologically conjugate to the shift map on infinitely many symbols. 
\end{thmA}

\begin{thmB}  If $B_0$ does not contain an asymptotic value, all the Fatou components are simply connected.   In particular, there are no Herman rings.
\end{thmB}

\begin{thmC}  If $f_{\la}$ is hyperbolic and $B_0$ does not contain an asymptotic value, the Julia set is connected and locally connected.  
\end{thmC}

The second set of results shows that the parameter plane has a structure analogous to that of the polynomials $P_a$.

\begin{thmD} The bifurcation locus divides the parameter plane into three types of hyperbolic components:  shell components, a central capture component or shift locus, and non-central capture components.   The shift locus is a punctured disk, but the remaining components are simply connected.   There are exactly $2q$ unbounded shell components and no unbounded capture components.   
\end{thmD}

The  properties of the  shell components are analyzed in \cite{FK} and \cite{CK1}.  In particular, they are simply connected and have piecewise analytic boundaries.  Here we concentrate on the properties of the capture components and their boundaries.  We prove
\begin{thmE} The shift locus contains the disk of radius $(\pi/4)^{1/q}$ centered at the origin.  There are only finitely many hyperbolic components with diameter greater than a given $\epsilon >0$.  
\end{thmE}

A parabolic point is a parameter for which the function has a parabolic cycle and a Misiurewicz point is a parameter for which an asymptotic value of the function is eventually periodic.  Such a function is not hyperbolic.  
For the polynomial families it is known (\cite{M2, Roe, EMS} and the references therein) that these types of points are landing points of curves in the shift locus and that they are dense in its boundary.   The techniques used to prove these results depend heavily on the maps having finite degree.  Our final result generalizes these results to our family of infinite degree maps. 

\begin{thmF}  The boundary of the shift locus contains a  a dense set of  parabolic  and Misiurewicz points  and they accessible; that is,  landing points of curves.   The boundaries of the non-central capture components contain a dense set of Misiurewicz points  which are accessible but they contain no parabolic points. \end{thmF}

The paper is organized as follows.  After defining our basic concepts and setting notation we introduce the family $f_{\la}=\la \tan^p z^q$ in section~\ref{family} and prove the results that constitute theorems A and B.   In section~\ref{hyper} we discuss hyperbolicity in our context and in section~\ref{Julia dichotomy} we prove the theorems whose results we have combined in theorem C.  We then turn to the parameter plane and in section~\ref{param} we show how the symmetries of the functions are reflected there.    We also recall earlier  results that describe the structure of the shell components.   In section~\ref{capture} we prove our results about capture components and their boundaries including theorem E and the boundary accessibility results in theorem~F.   The last section contains results about the bifurcation locus and the rest of theorem E.   
 
\section{Basics and Tools}
\label{basics}
 \subsection{Meromorphic functions.} Let $f:\mathbb{C}\to \widehat{\mathbb{C}}$ be a transcendental meromorphic function and let $f^n(z)$ denote the $n-$th iteration of $f$, that is $f^n(z)=f(f^{n-1})(z)$ for $n\geq 1$. Then $f^n$ is well-defined except at the poles of the functions $f, f^2, \ldots, f^{n-1}, \ldots$, which form a countable set. The dynamics divide the plane into two complementary sets, the \emph{Fatou set} where the dynamics are stable and the \emph{Julia set} where the dynamics are unstable or chaotic.  More precisely,   the Fatou set $F(f)$ of $f$ is
$$\{ z\in \mathbb{C} \ | \  \text{ for all } n,  f^n \text{ is defined  and normal in a neighborhood of } z\};$$  the Julia set $J(f)=\widehat{\mathbb{C}}\setminus F(f)$.  It contains infinity and all the poles of $f^n$. 

If $f$ is a meromorphic function with more than one pole, then the set of prepoles $\mathcal{P}=\cup_{n\geq 1}f^{-n}(\infty)$ is infinite. By the big Picard theorem, $f^n$ is normal on $\widehat{\mathbb{C}}\setminus\overline{\mathcal{P}}$, and $J(f)=\overline{\mathcal{P}}$, (see e.g. \cite{BKL1}). In this case, the Fatou set is the interior of the set where $f^n(z)$ is well-defined for all $n\geq 1.$

A point $c\in \mathbb{C}$ is called a \emph{critical point} of a meromorphic function, if $f'(c)=0$ or it is a pole of order greater than $1$. The image of a critical point is called a \emph{critical value} of $f$. A point $v\in \widehat{\mathbb{C}}$ is called an \emph{asymptotic value} if there is a path $\gamma: [0,\infty)\to \mathbb{C}$ such that $$\lim_{t\to \infty} \gamma(t)=\infty \text{ and } \lim_{t\to \infty} f(\gamma(t))=v.$$ For an asymptotic value $v$, if there is a simply connected unbounded domain $A$ such that $f(A)$ is a punctured neighborhood of $v$ and $f|_A$ is a universal covering, then $A$ is called an \emph{asymptotic tract} of $v$ and $v$ is called a \emph{logarithmic singularity}. For example if $f(z)=\tan z$. It is easy to check that $$\lim_{y\to \infty} \tan ( iy)=i\ \text{and}\ \lim_{y\to -\infty} \tan ( iy)=-i,$$ therefore $\pm i$ are asymptotic values of $f(z)$. In fact, they are logarithmic singularities with asymptotic tracts contained in the upper and half plane, respectively.

The set of \emph{singular values }of $f$, denoted by $S(f)$, is the closure of the set of  points in $\CC$ that are asymptotic or  critical  values of $f$.  In fact, $S(f)$ consists of those values at which $f$ is not a regular covering, that is $$f:\mathbb{C}\setminus f^{-1}(S(f))\to \widehat{\mathbb{C}}\setminus S(f)$$ is a covering.\footnote{Rational maps are defined at infinity so they may have critical values there.}
The set 
$$PS(f)=\overline{\cup_{v \in S(f)}\cup_{n=0}^\infty f^n(v)}$$
 is called the \emph{post-singular set} of $f$.

\begin{defn}
A meromorphic function $f$ is called a {\em singularly finite map if $S(f)$} is a finite set.
\end{defn}

A point $z$ is   a {\em periodic point of order $p\geq 1$}, if $f^p(z)=z$ and $f^k(z)\neq z$ for any $k<p$. The multiplier of the cycle is defined to be $\alpha=(f^p)'(z).$ The periodic point is \emph{attracting} if $0<|\alpha|<1$, \emph{super-attracting} if $\alpha=0$, \emph{parabolic} if $\alpha=e^{2\pi i \theta}$, where $\theta$ is a rational number, \emph{repelling} if $|\alpha|>1$.

Just as for rational maps, the repelling periodic points are dense in the Julia set, (see \cite{BKL1}). 
 
 \subsection{Fatou or stable sets of meromorphic functions}

Let $D$ be a component of the Fatou set, then $f(D)$ is either a component of the Fatou set or a component missing one point. For the orbit of $D$ under $f$, there are only two cases:
\begin{itemize}
\item there exists integers $m\neq n\geq 0$ such that $f^m(D)\subset f^n(D)$, and $D$ is called \emph{eventually periodic};
 \item for all $m\neq n$, $f^n(D)\cap f^m(D)=\emptyset$, and $D$ is called a \emph{wandering domain}.
\end{itemize}

Suppose that $\{D_0, \cdots, D_{p-1}\}$ is periodic, then either:
\begin{enumerate}
\item $D_i$ is (super)attractive: each $D_i$ contains a point of a periodic cycle with multiplier $|\alpha|<1$ and all points in each $D_i$ are attracted to this cycle. Some domain
in this cycle must contain a critical or an asymptotic value. If $\alpha=0$, the critical point itself belongs to the periodic cycle and the domain is called superattractive.
\item $D_i$ is parabolic: the boundary of each $D_i$ contains a point of a periodic cycle with multiplier $\alpha=e^{2\pi i p/q}$, $(p,q)=1$, and all points in each domain $D_i$ are
attracted to this cycle. Some domain in this cycle must contain a critical or an asymptotic value.
\item  $D_i$ is a Siegel disk: each $D_i$ contains a point of a periodic cycle with multiplier $\alpha=e^{2\pi i \theta}$, $\theta$ irrational. There is a      holomorphic homeomorphism
mapping each $D_i$ to the unit disk $\DD$, and conjugating the first return map
$f^p$ on $D_i$ to an irrational rotation, $z \mapsto \alpha z$, of $\DD$. The preimages under this conjugacy
of the circles $|\xi|=r, r<1$, foliate the disks $D_i$ with $f^p$ forward invariant
leaves on which $f^p$ is injective.
\item $D_i$ is a Herman ring: each $D_i$ is holomorphically homeomorphic to a round 
annulus and the first return map is conjugate to an irrational rotation of
the annulus by a holomorphic homeomorphism. The preimages under this conjugacy of the circles $|\xi|=r, 1<r<R$, foliate the disks with $f^p$ forward invariant leaves on which $f^p$ is injective.
\item  $D_i$ is an essentially parabolic (Baker) domain: the boundary of each $D_i$
contains a point $z_i$ (possibly $\infty$), $f^{np}(z)\to z_i$ for all $z\in D_i$, but $f^p$ is not
holomorphic at $z_i$. If $p=1$, then $z_0=\infty$.
\end{enumerate}

\begin{thm}\label{NoWD}\cite{BKL4}
If $f$ is a singularly finite map, then there are no wandering domains in the Fatou set.
\end{thm}

\begin{thm}\label{NoBD}\cite{Ber,RS}
If $f$ is a singularly finite map, then there are no Baker domains in the Fatou set.
\end{thm}

\subsection{Superattractive fixed points}  We will need the following theorem that characterizes the behavior of a holomorphic map near a superattracting fixed point (see \cite{CG},\cite{M2} for details). 
\begin{thm}\label{Bottchercoor}[The B\"ottcher Coordinate]
Suppose that $f$ is a holomorphic map defined in some neighborhood $U$ of  $0$ and that $0$ is a super-attracting fixed point; that is, $$f(z)=a_mz^m+a_{m+1}z^{m+1} +\cdots, \text{\ where\ } m\geq 2 \text{\ and\ } a_m\neq 0.$$ Then there exists a conformal map $\phi(z)$, called the {\em B\"ottcher coordinate at $0$}, such that $\phi\circ f\circ \phi^{-1}(z)=z^m$; it is unique up to the choice of an  $(m-1)$th root of unity.
 \end{thm}
 
In practice, we usually conjugate by a linear transformation so that $a_m=1$ and the map is monic.
Then,  following \cite{CG},  Theorem $4.1$,  the B\"ottcher map is defined by 
$$\phi(z)=z\prod_{n=0}^\infty (\frac{f^{n+1}(z)}{( f^{n}(z))^m })^{m^{-(n+1)}},$$
 where $f^0(z)=z$. Because we have assumed $f$ monic the B\"ottcher map  is unique and satisfies $$\lim_{z\to 0} \phi(z)/z\to 1.$$
 
Given $\phi$ defined on $U$, one would like to use the functional relation $\phi(f(z))=(\phi(z))^m$ to extend its definition to some maximal domain in the basin of attraction of $0$, and in fact to the entire immediate basin.   This, however, is not always possible.  Defining such an extension involves computing expressions of the form $z\to \sqrt[m]{\phi(f(z))}$ and this does not work in general since the $m$-th root cannot be defined as a single valued function. For example,  if $f(z)=0$ for some  $z \neq 0$ in the basin,   or if the basin is not simply connected. In fact,  $\phi$ can be extended until the pullback meets a singular point  -- and hence to the whole basin  if it contains no singular points.   

\section{The family $\mathcal{F}=\{f_\lambda=\lambda \tan^p z^q$ \}}
\label{family}
The focus of this paper is the family $$\mathcal{F}=\{f_\lambda=\lambda \tan^p z^q, \lambda\in \mathbb{C}\setminus \{0\}, p, q\in \mathbb{N}, pq>1 \}.$$
If $pq=1$, we obtain  the tangent family which has been well studied;  see \cite{KK1,CJK1,CJK2,K}.  It is the paradigm for families of functions with two asymptotic values and no critical values.   When $pq>1$, however, the functions in $\calf$ have a superattractive point at the origin which creates a substantive change in the dynamics.

\subsection{Covering properties of $f_\lambda$}\label{cov} For any $f_\lambda(z)=\lambda\tan^p z^q$, its derivative is $f_{\la}'(z)=pq\lambda z^{q-1}\tan^{p-1} z^q \sec^2 z^q$. Therefore solutions of $z^q=k\pi$ for all $k \in \mathbb{Z}$, are critical points, and their images are the critical value $0$;  moreover, each solution $z^q=k\pi+\pi/2$ for $k\in \mathbb{Z}$ is a pole of $f_\lambda(z)$ of order $p$. That is, if $p\geq 2$, $\infty$ is  a critical value as well as an essential singularity. Since $\tan z$ has two asymptotic values $\pm i$ and two asymptotic tracts, $f_{\la}$ has asymptotic values  $v_{\la}=i^p \la$ and $v_{\la}'=(-i)^p \la$ and has $2q$ asymptotic tracts each of which is mapped to a punctured neighborhood of an asymptotic value. The tracts are separated by the Julia directions along which the poles approach infinity.     If $p$ is odd,   $v_{\la}'=-v_{\la}$ and each has $q$ asymptotic tracts, whereas if $p$ is even, $v_{\la}'=v_{\la}$ and all $2q$ asymptotic tracts correspond to this asymptotic value. 
The singular set $S(f_\lambda)=\{0, v_{\la}, v_{\la}',  \infty\}$ if $p\geq 2$ or $\{0, v_{\la}, v_{\la}' \}$ if $p=1$.  For readability in the rest of this part of the paper we suppress the dependence on $\la$ and set $f=f_{\la}, v=v_{\la}$, etc. 

\subsection{Symmetries of $f$}\label{dyn symmetries}

For any $f \in \calf$ it is easy to see that 
 if $pq$ is even, $f(-z)=f(z)$ whereas if $pq$ is odd $f(-z) = -f(z)$.  

Let $\omega_j=e^{\pi j/q}$,  $j=0, \ldots,2q-1$, denote the $q^{th}$ roots of $\pm 1$.  The labelling is such that if $j$ is even, $\omega_j$ is a root of $+1$ and if $j$ is odd, it is a root of $-1$.  If  $p_k=|(k\pi +\pi/2)^{1/q}|$, the poles of $f$ are 
\[  p_{k,j}=(p_k)^{1/q}\omega_j, \,  j \mbox{ even and } p_{-k,j}=(p_k)^{1/q}\omega_j,  \, j \mbox{ odd }\, k \in \NN, j=0, \ldots, 2q-1. \] 
The poles  lie on rays that define the Julia directions for $f$ and divide the dynamic plane into $2q$ sectors of width $\pi/q$.     

In addition, with the same indices, the zeroes of $f$ are 
\[ 0,  q_{k,j}=|(k\pi)^{1/q}|\omega_j, \, j  \mbox{  even and } q_{-k,j}=|(k\pi)^{1/q}|\omega_j, \, j \mbox{ odd. } \]    They lie on the same rays as the poles.  Denote the Julia ray through $\omega_j$ by   $\delta_j(t)$, $t \in [0, \infty)$.

This tells us that in the dynamic plane of $f$ if $p$ is even, there is a $2q$-fold symmetry:   $f(\omega_j z)=f(z)$ for all $j$;  and if $p$ is odd,  there is a $q$ fold symmetry: $f(\omega_j z)=f(z)$ for $j$ even and $f(\omega_j z)=-f(z)$ for  $j$ odd.   
 In  each of the $2q$ sectors bounded by the Julia rays, there is an asymptotic tract.  Any curve $\gamma_j(t)$, $t \in [0, \infty) $, asymptotic to a ray which is not a Julia ray,  is an asymptotic path whose image lands at an asymptotic value.  If $p$ is odd, the asymptotic tracts of $v$ and $v'$ alternate.

\subsection{The basin of zero}
\label{basin of zero}

Denote the basin of attraction of zero by 
\[ B =\{z\ | \ f^n(z) \to 0 \ \text{as} \ n\to \infty \} \]
and let  $B_0$ be the immediate basin; that is the component of $B$ containing $0$.  Since zero is a superattracting fixed point, the basin is never empty.

\begin{lemma}\label{sym1}
If $z\in B $, then $-z\in B$.
\end{lemma}

\begin{proof} The orbits of $z$ and $-z$ are the same if $pq$ is even or are symmetric with respect to $0$ if $pq$ is odd. Thus either both of them belong to $B$ or neither does.
\end{proof}

\begin{lemma}\label{sym2}
$B_0$ is symmetric with respect to $0$. That is, if $z\in B_0$,  then  $-z \in B_0$ too. 
\end{lemma}
\begin{proof}   Let $-B=\{-z \ | \  z\in B \}$ and $-B_0=\{-z \ | \  z\in B_0\}$. It is obvious that  $0\in -B$ and that $-B=B$ by  lemma \ref{sym1}. Since $B_0$ is a component of  $B$,  $-B_0=B_0$.  \end{proof}

From this lemma, it follows that when $p$ is odd and $f$ has two distinct asymtotic values, then either both   are in $B_0$ or neither is.

 In order to determine when $B_0=B $, we need the following lemma.  
 \begin{lemma}\label{preimageSC}
 Let $U\in \mathbb{C}$ be a simply connected open set.
 \begin{enumerate}

 \item If  $0, v,v' \notin U$, then $f^{-1}(U)$ is a union of infinitely many simply connected open sets, and each component is conformally equivalent to $U$.  Note that if  $v, v'  \not\in \overline{U}$ then even if  $U$ is unbounded, the components of $f^{-1}(U)$ are bounded.
  \item If $0 \notin U$ but either $v \in U$ or $v' \in U$, then $f^{-1}(U)$ is a union of  infinitely many bounded simply connected sets and finitely many   simply connected unbounded sets, each of which is contained in one of  the asymptotic tracts.  If $p$ is even ($v'=v$) there are $2q$ unbounded sets.  If $p$ is odd, there are $q$ such sets,  $v'$ is in $-U$ and $f^{-1}(-U)$ contains another $q$ unbounded sets.  
   \item If $0\in U$ but  $v, v' \notin U$, then  $f^{-1}(U)$ is a union of infinitely many simply connected open sets,   each of which is a branched cover of $U$, branched at a single point.   The order of branching depends on $p,q$: the order of all components not containing $0$ is $p$ and that of the component containing $0$ is $pq$.

  \item If $0,  v,v'  \in U$, then $f^{-1}(U)$ is an unbounded connected set with infinite connectivity.
 \end{enumerate}
\end{lemma}
\begin{proof}
Write  $f(z)=\lambda \tan^pz^q=L\circ P\circ T\circ Q(z)$, where $L(z)=\lambda z$, $P(z)=z^p$, $T(z)=\tan z$, and $Q(z)=z^q$. Let $U$ be a simply connected open set in $\mathbb{C}$.

Case (1): If  $U$  contains neither  $0$ nor both   $v$ and $v'$, then $P^{-1}\circ L^{-1}(U)=\cup_{i=1}^p V_i$,  where each $V_i$ is a simply connected open set and none of the $V_i$ contains  either $\pm i$ or $0$. Moreover, $L\circ P: V_i\to U $ is conformal. Recall that the map $$T: \mathbb{C} \to \widehat{\mathbb{C}}\setminus \{\pm i\}$$ is a universal covering map. Since each $V_i\subset \mathbb{C}\setminus \{\pm i\}$, it follows that $T^{-1}(V_i)=\cup_{j=1}^\infty W_i^j$, where each $W_i^j$ is a simply connected open set and $T: W_i^j\to V_j$ is also conformal. Moreover, each $W_i^j$ is contained in a vertical strip of width $\pi$:$$H_j=\{z=x+yi \ | \ a_{j-1}\leq x\leq a_j, \text{ and } a_j-a_{j-1}=\pi\}.$$ Note that $0\notin V_i$ because $0\notin W_i^j$.   Therefore for each $W_i^j$, $Q^{-1}(W_{i}^j)=\cup_{k=1}^q U_{i}^{j_k}$, where the $U_i^{j_k}$ are simply connected open sets and  $Q: U_i^{j_k}\to U_i^j$ is also conformal.

Case (2): If $v  $ or $v' $ is in $U$ but $0$ is not, then $P^{-1}\circ L^{-1}(U)=\cup_{i=1}^p V_i$, where  each $V_i$ is a simply connected open set not containing $0$. Moreover, $L\circ P: V_i\to U $ is conformal. As these $q$ sets are permuted by  the rotation $\tau (z)=e^{2\pi i/p}z,$ the asymptotic value of $T$, $\pm i $,  are in different $V_i$'s.  As above,  the map $$T: \mathbb{C} \to \widehat{\mathbb{C}}\setminus \{\pm i  \}$$ is a universal covering map so that if neither $i$ nor $-i$ is in  $V_i$, $T^{-1}(V_i)=\cup_{j=1}^\infty W_i^j$, where each $W_i^j$ is a simply connected open set and $T: W_i^j\to V_j$ is   conformal.  Again each $W_i^j$ is in a vertical strip of width $\pi$, $$H_j=\{z=x+yi \ | \ a_{j-1}\leq x\leq a_j, \text{ and } a_j-a_{j-1}=\pi\}.$$  If    $i$   is in  $V_i$, it has no preimages under $T$ and  $T^{-1}(V_i)$ consists of a single  simply connected unbounded set  $W_i$ and  $T: W_i\to V_i\setminus \{i\}$,  is a universal covering;  similarly if $-i$ is in $V_i$.  Note that $0\notin W_i^j$ or $W_j$ because  $0\notin V_i$. Therefore for each $W =W_i^j \text{ or } W_i$, $Q^{-1}(W)$ is a union of $q$ simply connected open sets.

Case (3): If $0$ is in  $U$ but both  $v$ and $v'$ are not, then $P^{-1}\circ L^{-1} (U)=V$ is a simply connected set containing $0$ but neither $i$ nor $-i$, and  $L\circ Q: V\to U$ is a branched cover of degree $p$. Similarly $T^{-1}(V)$ is also a union of infinitely many simply connected sets in vertical strips, each  of  which contains a preimage of $0$. Then for each component $W$ of $T^{-1}(V)$, either $0\not\in W$ and $Q^{-1}(W)$  is a union of $q$ simply connected sets containing a single non-zero preimage of zero or $0\in W$  and it is one simply connected set.

Case (4): if $0, v,v'  \in U$, then $P^{-1}\circ L^{-1} (U)=V$ is a simply connected set containing $0$ and both $i$ and $-i$.  If $U$ is bounded, the set $V'=\widehat{\mathbb{C}}\setminus V$  is a simply connected unbounded set in $\hat\CC$ that doesn't contain $0, i$ or $-i$.   As in Case (1), we apply the map $T^{-1}$ and $T^{-1}(V')$ is a union of simply connected sets $W_i$ contained in vertical strips, but in this case, $T:W_i \rightarrow V'$ is a branched cover of degree $p$.   If $U$ is unbounded, $V'$  contains  $p$ simply connected components and we apply $T^{-1}$ to each of these.   In either case,  $(Q \circ T)^{-1}(V')$ is a union of infinitely many bounded simply connected sets so that  $f^{-1}(U)$ is a connected set with infinite connectivity.
\end{proof}

\begin{lemma}\label{allin}  The immediate basin 
   $B_0 $ contains an asymptotic value if and only if all the preimages of $0$ are in $B_0$.
\end{lemma}
\begin{proof}
Let $D_r(0)\subset B_0$ be a small disk centered at $0$ with radius $r$, such that no point in the orbit of asymptotic values lies on $\partial D_r$. For any $n\geq 1$, let $U_n$ be the component of $f^{-n}(D_r)$ containing $0$. It follows that  $U_{n-1}\subset U_n$ and $B_0=\cup_{n\geq 0}U_n$.  By lemma \ref{preimageSC}, $U_n$ is simply connected and contains no other preimages of $0$ if and only if $v,v' \notin U_{n-1}$.
\end{proof}

\begin{remark}
Lemma~\ref{allin} indicates that when we take successive preimages of a neighborhood of $0$ by an inverse branch of $f$ that fixes $0$,  they are nested and, if for some $n$, $U_n$ contains an asymptotic value, but $U_{n-1}$ does not, then $U_n$ does not contain any non-zero preimages of $0$; note these  are critical points. Therefore if there is no asymptotic value in $B_0$, the B\"ottcher map of  theorem~\ref{Bottchercoor} defined on a neighborhood in $B_0$ extends injectively to the whole of $B_0$ whereas if   $B_0$ does contain an asymptotic value,   the B\"ottcher map extends injectively to the $U_n$ containing that asymptotic value, but not to $U_{n+1}$;   in particular it is defined at the asymptotic value.  We will describe this further in section~\ref{Julia dichotomy}.  \end{remark}

Since  theorem~\ref{Bottchercoor} applies to monic maps, we conjugate $f(z)=f_{\la} (z)=\lambda \tan^p z^q$ to the monic map
\[ h_{\mu}(z)=\mu^{pq} \tan^p (z/\mu)^q=z^{pq}+\cdots \] 
   by the linear map $\sigma(z)=\mu z$, where $\mu=\lambda^{1/(pq-1)}$ is chosen as a $(pq-1)^{st}$ root of $\la$.   Then we have 
\begin{thm}\label{Bott}[The B\"ottcher coordinate for $f_{\la}$] If $U$ is a neighborhood of $0$ in $B_0$, there is a   conformal map, 
$\phi$, defined up to the  choice of  a $(pq-1)^{st}$ root of $\lambda$, such that for $z \in U$,  $\phi(0)=0$ and $\phi(f(z)) = (\phi(z))^{pq}$.
\end{thm}

\begin{proof} By  theorem~\ref{Bottchercoor},  on a small neighborhood $U$ of $0$ there is a map $\phi_h:U \rightarrow \DD$ such that $\phi_h(h(z))=(\phi_h(z))^{pq}$.  Therefore,   $\phi(z)=\phi_{\la}(z)=\phi_h(\sigma(z))$   is the desired map. 
  \end{proof}

Note that the asymptotic values are in the immediate basin of $0$ of $h_\mu$ if and only if the asymptotic values of $f_{\la}$ are in $B_0$ so  the map $\phi$   extends  injectively to  the preimage of the disk with radius $|\phi(v)|$ in $\DD$.

\begin{thm}\label{dichotomy}
The immediate basin of $0$ is completely invariant, that is $B=B_0$,  if and only $B_0$ contains an asymptotic value. \end{thm}

\begin{proof}
By lemma \ref{allin}, if $B_0$ contains an asymptotic value, it also contains  all the preimages of zero  so it  is completely invariant.

On the other hand,  again by  lemma \ref{allin}, if $B_0$ contains no asymptotic value,  then the only preimage of $0$  in $B_0$ is $0$.  Therefore all the other preimages are  in other distinct components of $B$.
\end{proof}

\begin{lemma}\label{simpconn}
If $B_0$ does not contain an asymptotic value,  all the components of $B $ are simply connected. Moreover, if $v' =- v$, they are each in different components.
\end{lemma}
\begin{proof}
By  lemma \ref{allin}, $B_0=\cup_{n\geq 1} U_n$. Since  $B_0$ doesn't contain an asymptotic value, it follows that the B\"ottcher map extends to all of $B_0$ and so it is simply connected. Applying  lemma \ref{preimageSC}, we know any preimage, $f^{-n}(B_0)$, is also simply connected.  

Suppose both $v' = -v$ and $v \in U$, where $U$ is some preimage of $B_0$.  Then by lemma~\ref{sym1}, $ -U$ is also a preimage of $B_0$ and it contains $v'$ and $v$.  Therefore $U \cap -U \neq \emptyset$  and   $0 \in U$ so that $U=B_0$.   \end{proof}

\begin{figure}[htb!]
\begin{center}
\includegraphics[height=3in]{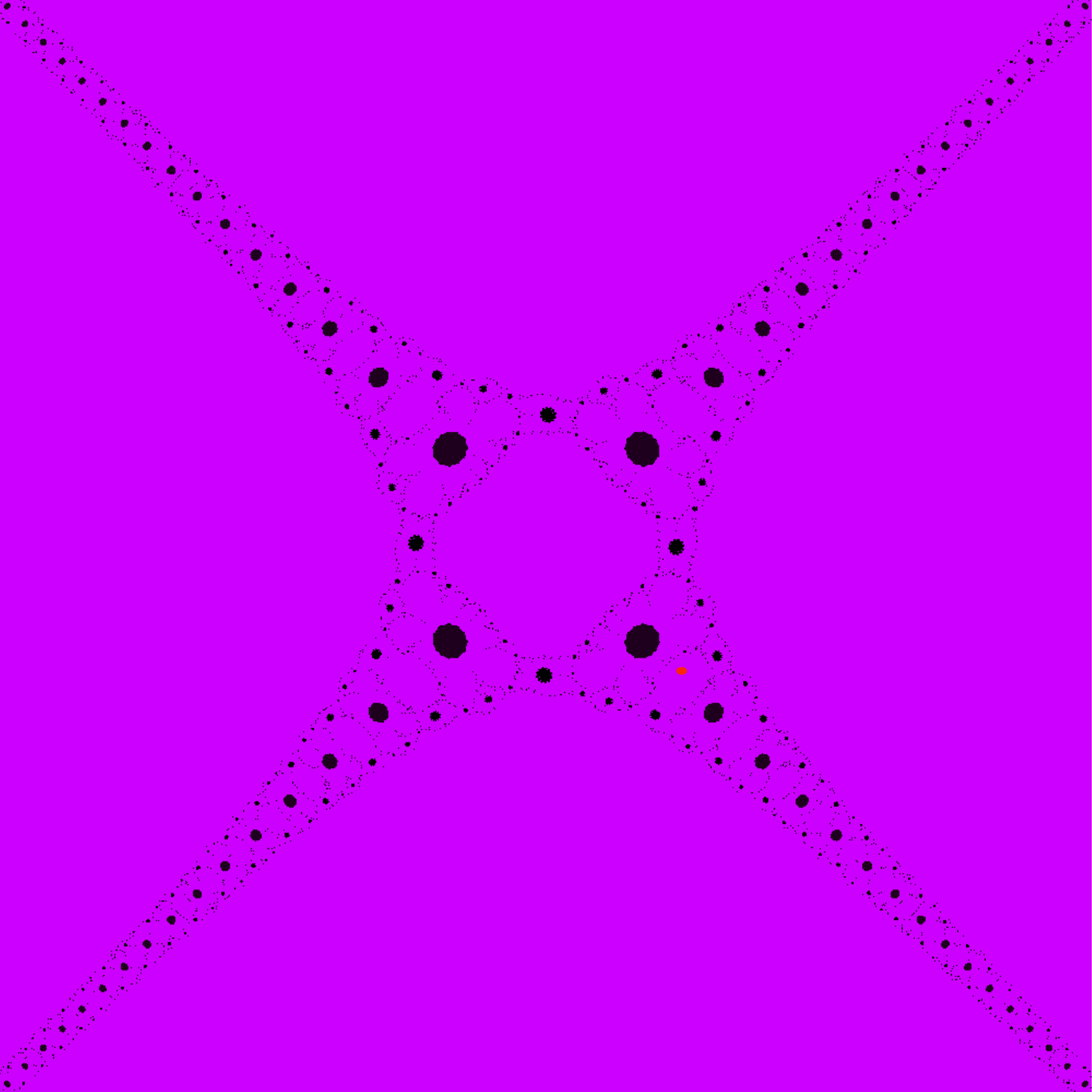}

\end{center}
\caption{The dynamic plane for $(1.22+1.3i) \tan^3 z^2$. The asymptotic value lands in the immediate basin of zero after one iteration. The magenta regions are the basin of $0$.   The  black dots are the prepoles that accumulate to form the boundaries of the components of the basin. The asymptotic value is in the component between the first and second poles in the fourth quadrant.  }
\label{per1 dyn}
\end{figure}

 \begin{figure}[htb!]
\begin{center}
\includegraphics[height=3in]{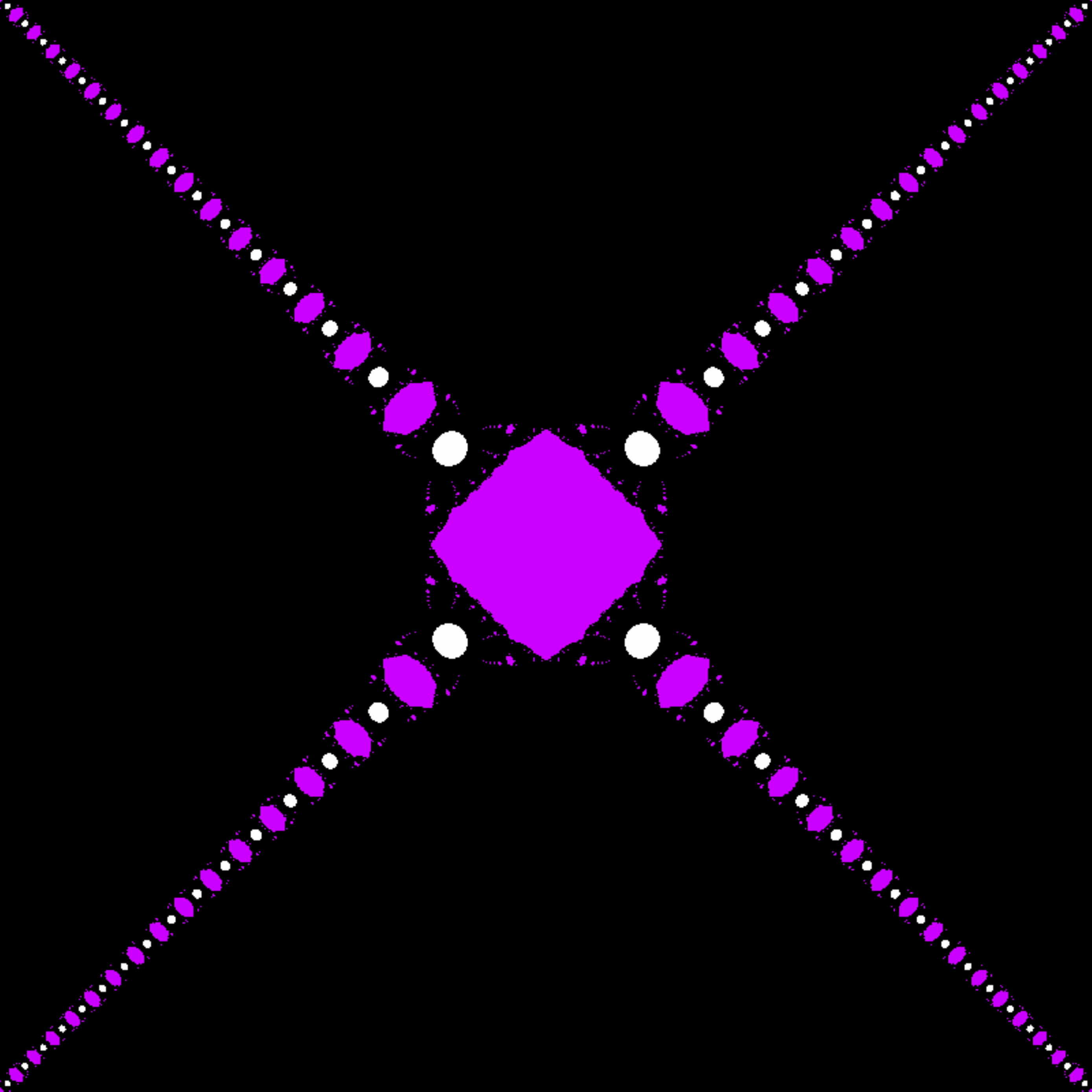}

\end{center}
\caption{ The dynamic plane of $2i \tan^3 z^2$ which has a non-zero attracting fixed point. The black region is the basin of a period $1$ cycle. The magenta regions are the basin of $0$. The white dots surround the poles}
\label{per1 central}
\end{figure}

\begin{thm}\label{ubddcomps}  If the asymptotic values belong to the Fatou set or if they are accessible boundary points of a Fatou component  there are unbounded components of the Fatou set.    
\end{thm}

 \begin{figure}[htb!]
\begin{center}
\includegraphics[height=4in]{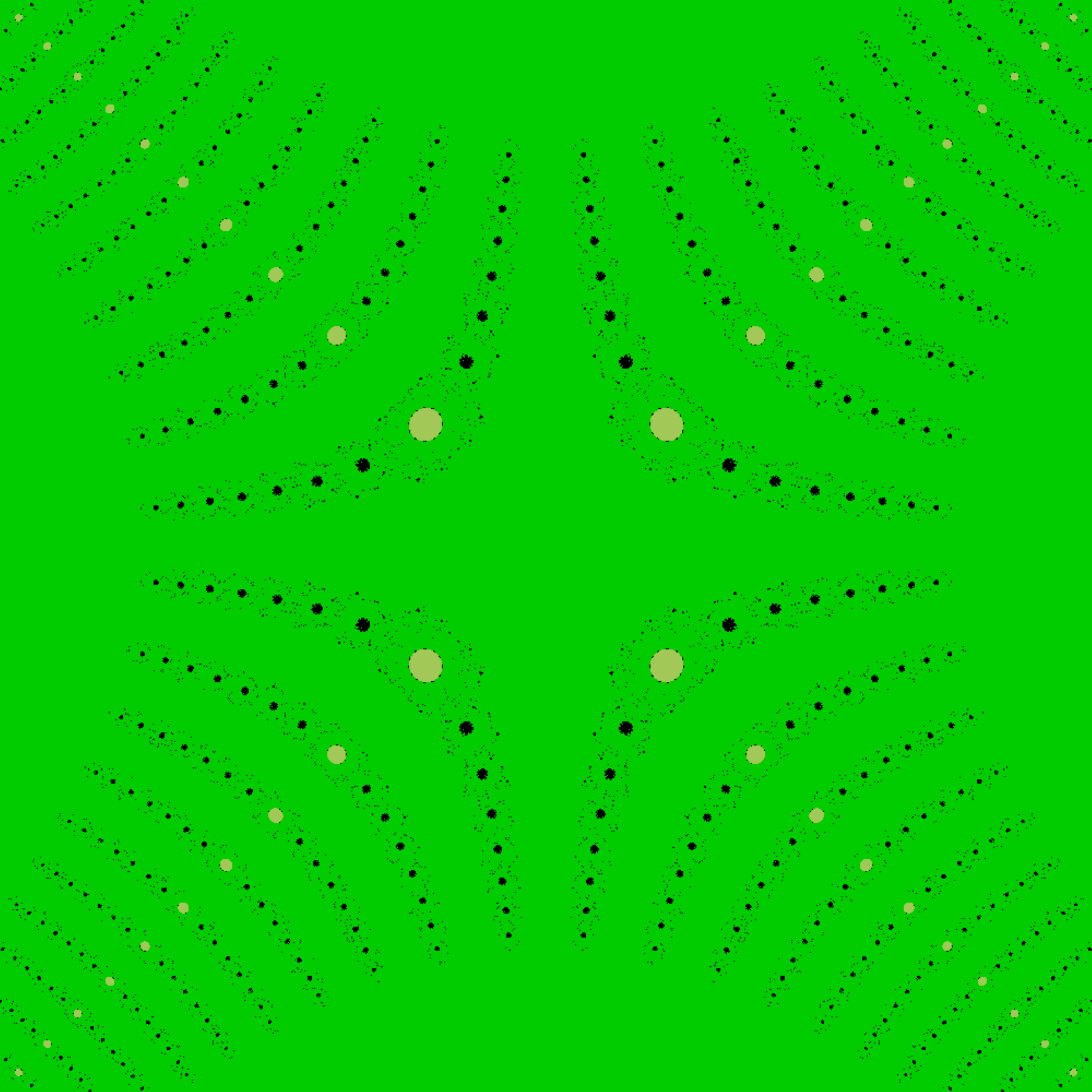}

\end{center}
\caption{The dynamic plane for $f(z)=\sqrt{\frac{\pi}{4} }\tan^3 z^2$. The asymptotic value $v$  is $-\sqrt{\pi/4}i$, $f^3(v)=f^2(v)$ and $v$ is on the boundary of the immediate basin of zero. The full basin contains infinitely many components, both bounded and unbounded.  Roundoff error has truncated the prepoles in the boundary that tend to infinity; in fact, the hyperbola-like curves containing prepoles around which are bounded domains extend to infinity separating the asymptotic tracts into infinitely many unbounded domains. 
}
 
\label{M-pt dyn}
\end{figure}

 \begin{figure}[htb!]
\begin{center}
\includegraphics[height=4in]{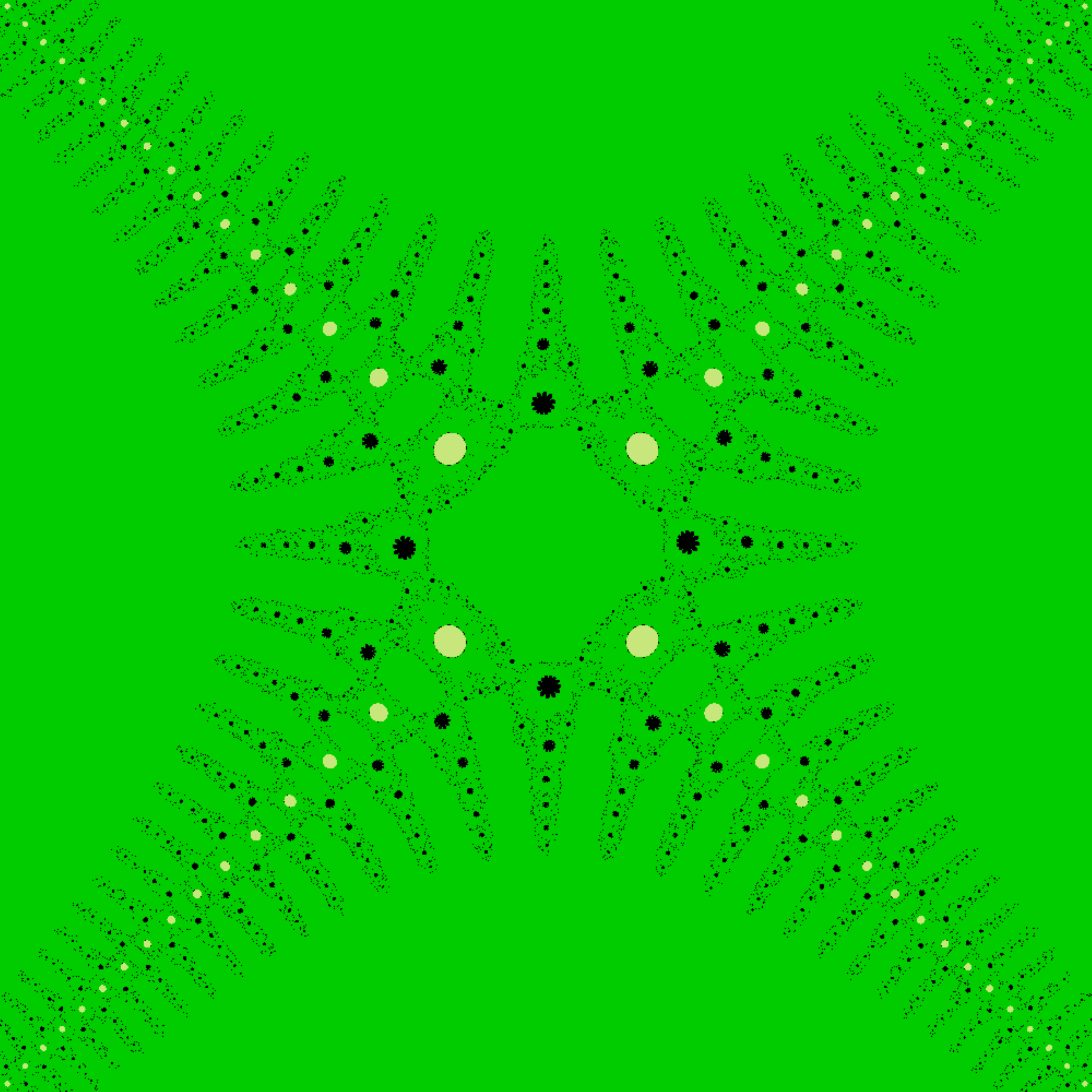}

\end{center}
\caption{The dynamic plane for $f(z)=\la \tanh^3z^2$ where $\la=1.067098-1.135274i$.   The asymptotic value $v$  is on the boundary of a preimage $B_1=f^{-1}(B_0)$ in the fourth quadrant.  Again the full basin contains infinitely many bounded and unbounded components and divides the asymptotic tracts into infinitely many regions.  
 }
\label{M-ptnoncent dyn}
\end{figure}

\begin{proof}   Let $V$ be a punctured neighborhood of the asymptotic value $v$ so that $f:  f^{-1}(V) \rightarrow V$ is  a regular covering and $\{ f^{-1}(V) \}$ contains $q$ or $2q$, as $p$ is odd or even,  unbounded simply connected components, each  contained in an asymptotic tract for $v$. If $p$ is odd, $v'=-v$, and $\{f^{-1}(-V)\}$ also has $q$ unbounded simply connected components.  

If $p=1$ or $2$, $v$ and $v'$ are omitted; then  these $2q$ unbounded components are the only components of $\{f^{-1}(V)\}$.    If $p>2$, $v$ and $v'$ are not omitted and $\{f^{-1}(V)\}$ contains infinitely many bounded domains each containing a preimage of $v$ or $v'$.  Since $V$ contains no critical value, $f$ is a regular covering of  each component of $\{f^{-1}(V)\}$ onto its image.

 Note that if $v$ is in the Fatou set, the neighborhood $V$ can be taken small enough that it, and hence all its preimages,   belong to the Fatou set.  Suppose this is the case and that  $B_0$ is not completely invariant.  Let $U$ be the component of the Fatou set  containing $V$;  by Cases  (1), (2) and (3) of  Lemma \ref{preimageSC},  $\{f^{-1}(U) \}$ consists of $2q$  unbounded components contained in the asymptotic tracts and infinitely many bounded components.  They are all  simply connected and so each unbounded component is   contained a single asymptotic tract.   
 
In figure~\ref{per1 dyn} the asymptotic value is contained in $B$ and in figure~\ref{per1 central} the asymptotic value is contained in the basin of a non-zero attracting point.

Now suppose  $v$ is not in the Fatou set but is an accessible boundary point of a component  $U$ of the Fatou set, then $-v$ is also accessible.  If $U=B_0$,   let $\gamma(t)$ be a symmetric curve in $B_0$ through $0$ that lands at both $v$ and $-v$.   If $U \neq B_0$, then $v$ and $-v$ are in disjoint symmetric components $U$ and $-U$ and for readability we work only with $U$. There are two possibilites:  either $U=f^{-n}(B_0)$ for some $n>0$ and some branch of the inverse, and  it is bounded contains a preimage of $0$; 
 or $U$ is in the basin of a periodic Siegel disk and it contains a preimage of the periodic point (or a periodic point) in the disk.     In either case we let $\gamma$ be a curve that starts at this preimage in $U$ and lands at $v$.  

Then $f^{-1}(\gamma(t))$ contains $\cup^{\infty}_{j=-\infty} \cup_{i=1}^m, \delta_i^j(t)$ where $m=q$ or $m=2q$ and $\{\delta^j_i\}$ is a collection of unbounded curves contained in the $i^{th}$ asymptotic tract for $V$. 
If $U=B_0$, $\delta_i^j$ and $\delta_{i+1}^j$ meet at a preimage of $0$ but if $U \neq B_0$, they are disjoint.  In either case, we see that each asymptotic tract intersects infinitely many components of the Fatou set.  

 If $p>2$, $f^{-1}(\gamma(t))$ also contains infinitely many bounded curves limiting on the preimages of $v$.    Note that since $v$ is a boundary point of $U$, if $N$ is a neighborhood of $\infty$,  each of the infinitely many components $W_j=f^{-1}(U)   \cap N $ is only  a subset of the asymptotic tracts and $f^{-1}(U)$ consists of infinitely many unbounded components. 
 
Figures~\ref{M-pt dyn} and~\ref{M-ptnoncent dyn}   show examples where the asymptotic value is respectively on the boundary of $B_0$ or the boundary of  $U=f^{-1}(B_0)$ for a branch of the inverse.   The bounded Fatou components extend all the way to infinity in the directions indicated.  This doesn't show because of round off error.

\end{proof}

\begin{remark}  Recall that if the asymptotic values belong to the Fatou set, either they are attracted to zero, and thus belong to $B$, or they are in attracting basin of, or accumulate on the boundary of a non-zero periodic cycle.
\end{remark}

 We also have 
 
 \begin{prop}\label{infaccess} If $V$ is an unbounded component of the Fatou set of  $f_{\la}$,  and $f(V)$ contains an asymptotic value, then infinity and all the prepoles are accessible points from inside the Fatou set. 
  \end{prop}
 \begin{proof}  
      Since $f(V)$ contains an asymptotic value,  $V$   must intersect at least one of the asymptotic tracts.  It follows that an asymptotic path in the tract maps to a path landing at an asymptotic value.  The asymptotic path gives access to infinity.   The asymptotic paths pull back to paths that give access to the prepoles.  
  \end{proof}

 \begin{remark}  Note that this does not imply that all  points of $J$ are accessible.  The prepoles are only a countable dense subset of $J$.
 \end{remark}

 We can now complete the classification of Fatou components for functions in $\calf$. \footnote{This theorem is proved for the family $ \la \tan z^2$ in \cite{Nandi}}
\begin{thm}\label{NOHR}
No  map $f_\lambda$ can have a  Herman ring.
\end{thm}
 
 \begin{proof} This theorem was proved in \cite{KK1} for the family $\la\tan z$ and there are two proofs for the family of functions with exactly two finite asymptotic values and one attracting  fixed point of constant (non-zero) multiplier in \cite{CJK2}. Our family reduces to the first if $pq=1$ and is similar to the second, although the asymptotic values agree or are symmetric and the constant multiplier is zero.   
 
 Suppose $f$ has a periodic cycle of Herman rings.  Let $A$ be one of the components of the cycle; then for some $k$,  $f^k(A)=A$.  Choose $z\in A$ and set $\gamma=\overline{\{f^{kn} (z)\}_{n=0}^{\infty}}$;  $\gamma$ is a topological circle that  bounds a topological disk $D(\gamma)$ containing the inner boundary of $A$; moreover $\gamma$ is forward invariant under $f^k$.  Let $\gamma_1=\gamma$ and for $i=1, \ldots, k$,  let $\gamma_i=f^{i-1}(\gamma)$; let $ D(\gamma_i)$ the disk bounded by $\gamma_i$.

The curves $\gamma_{i}$ separate  the annuli $A_i=f^{i}(A)$ into two annuli, an inner annulus $I_{i}$ contained in $D(\gamma_{i})$ and an outer annulus $O_{i}$, its complement.  
Since $A$ is part of a cycle of Herman rings $f: f^{N-1}(A) \rightarrow f^{N}(A)$ is a homeomorphism.  Moreover,  $f^k$ is conjugate to a rotation on $A$ and is an orientation preserving homeomorphism.

We now separate our considerations depending on how the asymptotic value(s) are situated with respect to the disks $D_i=D(\gamma_{i})$. 
\begin{itemize}
\item Suppose first that $v$ and $v'$ are outside all of the disks  $D_i$.   It follows that  for each $i$, $\{f^{-1}(  \gamma_i) \}$ consists of infinitely many bounded closed curves, one of which is $\gamma_{i-1}$.  Suppose $g_i: \gamma_i \rightarrow \gamma_{i-1}$ is this branch of  $f^{-1}$.  Then $g_i :D_i \rightarrow D_{i-1}$ and therefore, going forward,  $f: D_{i-1} \rightarrow D_i$.  This means that $f^n$ forms a normal family on the $D_i$.  This is a contradiction because  each $D_i$ contains the inner boundary component of $A_i$.

\item Now assume  $v \in D_i$ for some fixed $i$.    It follows  each  closedcomponent of $\{ f^{-1}(\gamma_i)\}$ is either an unbounded arc or a simply connected closed curve.   If $p=$ or $2$ these are unbounded curves are the only components of the inverse; if $p>2$,  all other components are bounded.   In the first case, there cannot be a Herman ring because the maps along the ring are all injective.   In the second, one of the bounded components is a $\gamma_i$ belonging to an annulus of the ring and the argument of the previous bullet applies. 

\item Finally assume that both asymptotic values belong to some $D_i$ so that 
  there are no asymptotic values in $\hat{\CC} \setminus \overline{D_i}$.  Then  $\{ f^{-1}(\hat{\CC} \setminus \overline{D_i}   \}$ consists of infinitely many pairwise disjoint bounded punctured topological disks $U_{m} \setminus \{p_{m} \}$ where the  $p_{m}$ are poles.  One of these, say $U_i$ intersects the inner annulus $I_i$.  It follows that $f:  I_i \rightarrow  O_{i+1}$.  
 
Since there are at most two asymptotic values,  if all the disks $D_j$, 
  $j= 1, \ldots k$, are mutually disjoint,  there are no asymptotic values either inside $D_j$, $j \neq i$ or in their common exterior.    It follows that for all $j \neq i$, $f:   I_j \rightarrow  I_{j+1}$,  and therefore 
  $f^k: I_i \rightarrow  O_i$.  This is a contradiction since $f$ is conjugate to a rotation on $A_i$.  
  
  Because the $\gamma_j$ are disjoint, if some of  the $D_j$ intersect,  they   form a nest.   Let $I_{j_1}$ be the innermost inner annulus and $O_{j_2}$  the outermost outer annulus.   Then the ring $R$ between $\gamma_{j_1}$ and $\gamma_{j_2}$ contains no asymptotic values and $f^{-k}$ maps $R$ homeomorphically to itself, preserving the inner and outer boundaries.   Now we argue as above to conclude that the $f^{nk}$ form a normal family on $R$.  This implies $j_1=j_2$,  the nest contains only $D_i$.  We saw above that in this case the map cannot be  conjugate to a rotation.  This contradiction completes the proof that are no Herman rings.  
  
  \end{itemize} 

\end{proof}

A corollary of the above discussion is 

\begin{cor}\label{Jconn}
If $B_0$ is not completely invariant, the Julia set is connected.  
\end{cor}

\begin{proof}
If $B_0$ is not completely invariant $v,v' \not\in B_0$ and by lemmas \ref{allin} and \ref{preimageSC},   the basin $B$ consists of  infinitely many bounded simply connected components.  

If $v,v' \in B \setminus B_0$, the basin is the full stable set so   the  Julia set is the complement of infinitely many simply connected components and is connected. 
 
Otherwise, either there are no other components of the Fatou set and again the Julia set is connected or there is a cycle of attracting, parabolic or Siegel domains attached to an asymptotic value.  By theorem~\ref{NOHR},  there are no Herman rings so any   components in such a cycle and their preimages are  simply connected, so again the Julia set is connected.  
\end{proof}

\section{Hyperbolic maps}
\label{hyper}
A general question in dynamical systems  is  what happens to the dynamical properties if we perturb a function slightly.  In particular, when do these properties persist under deformation;  if they do, the map is called {\em hyperbolic}.   More specific characterizations of hyperbolicity depend on the context.  The following discussion is a summary of material in \cite{DH, McM, M2} for rational maps.  Discussions of hyperbolicity for transcendental meromorphic maps that include singularly finite ones can be found in \cite{KK1, RS,Z}.  We omit proofs and direct the reader to the literature.   
 
 From the discussion below, the following  definition of hyperbolicity for rational maps is a good one to adapt to singularly finite maps.  
 
 \begin{defn}\label{hyperbolic} Suppose $f(z)$ is a singularly finite meromorphic map.  Then $f(z)$ is {\em hyperbolic} if $PS(f) \cap J(f) = \emptyset$.
 \end{defn}
 
 Note that if $f$ is a singularly finite hyperbolic map, the singular points must all be attracted to attracting or superattracting cycles.

     A related condition,(see \cite{M2}), which  is  equivalent to the above, says  that 
      a rational map is {\em expanding on its Julia set} if there exist constants $c>0$ and $K>1$ such that for all $z$ in a neighborhood $V \supset J(f)$,   $|(f^{n})'(z)|> c K^n$.

 Because the Julia set of a meromorphic function is unbounded and its iterates have singularities at the prepoles we need a version of this condition tailored to transcendental maps.  
 We use   the following one proved in \cite{RS} which applies to hyperbolic functions in $\calf$.

 \begin{prop}[Rippon-Stallard] \label{RS}   If $S(f)$ is bounded and $\overline{PS(f)} \cap J(f) = \emptyset$,   then there exist two constants  $c>0$ and $K>1$ satisfying 
 \[ |(f^n)'(z) | > c K^n(|f^n(z)|+1)/(|z|+1). \]
 for all $z \in J(f) \setminus A_n(f)$ and all $n$ where $A_n(f)$ is the set of points where $f^n$ is not analytic (prepoles of lower order).   \end{prop}
 Note that in the family  $\calf$ the poles  are critical points in the Julia set and their forward orbits are finite. 
   They and their preimages form the set $A_n(f)$.     

\section{Julia set dichotomy}
\label{Julia dichotomy}
The Julia set of a hyperbolic quadratic map is either locally connected or a Cantor set.  In this section, we prove the same is true for maps in the family $\calf$. 

\subsection{Local connectivity for hyperbolic maps}
In this section we prove
\begin{thm}\label{J is loc conn} 
If the Julia set of a hyperbolic map in $\calf$ is connected,  it is  locally connected.
\end{thm}
This is proved for rational maps in \cite{M2}, section 19, based on Douady's work.    The proof there follows directly from three lemmas.  Here we  adapt them to the family $\calf$ to obtain the proof for these maps.   In the following lemmas we assume the Julia set is connected and $f$ is hyperbolic.  As we saw above, this means that $B_0$ is not completely invariant, all components of the Fatou set are simply connected and the asymptotic value is either attracted to zero or a non-zero attracting cycle.

\begin{lemma}\label{key1} If $f$ is hyperbolic and $U$ is a simply connected component of the Fatou set, then $\partial U$ is locally connected. 
\end{lemma}

\begin{proof} 
The immediate basin of zero, $B_0$, is  bounded and forward invariant  and therefore contains no prepoles.  We can apply the argument in \cite{M2} section 19: we use the B\"ottcher map $\phi$ on $B_0$ to conjugate $f$ to the map $z \mapsto z^{pq}$ on the disk $\DD$ and define an annulus structure that tesselates $B_0$.   There are rays $R_{t}$ in $B_0$ that are the pullbacks by $\phi$ of radii of angle $t$ in $\DD$.  As in \cite{M2}, we fix an annulus $A_0$ in $B_0$ and look at the intersection $I_{t} = R_{t}\cap A_0$.  The branches of  $f^{-1}$ that fix $B_0$ define a set of  pullbacks of the $I_{t}$ whose lengths go to zero. 
  The Rippon-Stallard estimate applies to shows  that these pullbacks converge uniformly to $\partial B_0$ and therefore that it is locally connected.  

We divide the rest of the proof into two cases:

\begin{enumerate}
\item Suppose that the basin of zero, $B$, contains the asymptotic value. Then  the basin is the whole Fatou set.  If $U_n$ is any component of $B$ such that   $f^n(U_n)=B_0$,  $f^n:\partial U_n \rightarrow \partial B_0$ is a local homeomorphism and so its boundary is locally connected.  
   
\item Assume now that $f$ has a non-zero attracting periodic cycle.  
\begin{itemize}
\item 
 The boundaries of all bounded components of the basin of zero are locally connected since, as above, they contain no prepoles and argument above applies. 
\item 
 If $k=1$, there is only one component $V$ of the cycle.  Then it is unbounded,   contains a periodic point $\zeta$ and one asymptotic value $v$ .   Since $f$ is infinite to one on $V$, $\partial V$ is an infinite curve and it contains infinitely many poles.  We can conjugate $f$ on $V$ to a linear map  to $\CC$ in a neighborhood $\calo$ of $\zeta$ by the K\"onig map $\phi$ where $\phi(\zeta)=0$ and $\phi'(\zeta)=1$ and extend it injectively until we reach $v$. Suppose $\phi(v)=r_0 e^{2\pi it^*}$.  The rays in $V$ are the curves $R_t=\phi^{-1}(re^{2 \pi it})$ for fixed $t$ and varying $r$.   The ray $R_{t^*}$ joins  $\zeta$ to $v$.  Let $\gamma$ be the curve in $V$  defined by $|\phi(\gamma)|=r_0$; let  $A_0$ be the annulus between $\gamma$ and $f(\gamma)$.  It  is a fundamental domain for the action of $f$ on $V$.   Set $I_t=A_0 \cap R_t$.   
  
   There is a unique branch $g$ of $f^{-1}$ such that $\ell^*=g(R_{t^*})$  joins $\zeta$  to infinity.   The preimages  $\{l_j= f^{-1}(\ell^*) \}$ each join a pole $p_i$ on $\partial V$ to infinity and separate $V$ into strips.  Because the poles are accessible the $l_j$ extend continuously to them. 
    Starting with the strip containing the fixed point, the strips can be labelled consecutively and this can be used to enumerate the branches of the inverse in the same way one defines branches of the logarithm.  Their endpoints separate $\partial V$ into compact sets $J_i$ such that $J_i \cap J_j =p_i$.     The curves $\{ \gamma_i=f^{-1}(\gamma) \}$ are  U-shaped doubly infinite curves separating the preimages of $\zeta$ from one another.  The preimages of the $\gamma_i$ are curves in the strips that end at the prepoles.  
    
 For each possible choice of inverse branches of $f$, we obtain a pullback of $A_0$ and these form a tesselation of $V$.    Choosing sequences of inverse branches appropriately, we can form successive nests of these fundamental domains.  Each such nest is contained in one of the strips.  The inverse branches are hyperbolic isometries so the hyperbolic lengths of the pullbacks of the $I_t$ are bounded.  For each $t$, we can find of sequences of inverse branches to the pullbacks of the $I_t$ lie in a nested set of annuli.   If $t=t^*$, the pullbacks nest down to a prepole;  since these are accessible, they have a unique accumulation point.  For other values of $t$, the Rippon-Stallard estimate applies to show that the accumulation points are unique.  Therefore, the  boundary  of $V$ is locally connected. 
 
 \item
 If $k>1$, the argument is even easier.    Now the asymptotic value  is in a bounded component $V_1$.  Then $V_0=f^{-1}(V_1)$ is unbounded and fixed under $f^k$.  The argument above applies to $V_0$.   Since there are no critical values in the cycle of domains, the inverse maps  going backwards around the cycle are local homeomorphisms and    because the boundary of $V_0$ is locally connected they extend to the boundary.  This proves each of the domains in the cycle has a locally connected boundary. 
  \end{itemize}

\end{enumerate}

\end{proof}

\begin{lemma}\label{key2} Suppose that $f$ is hyperbolic with connected Julia set.     Given $\epsilon >0$, if $q>1$ the  Euclidean diameters only finitely many   components of the Fatou set  are greater than $\epsilon$.  If $q=1$, the conclusion holds in the spherical metric. 
\end{lemma}
 
\begin{proof}
 By hypothesis, the Fatou set  consists of infinitely many bounded components, and at most $2q$ unbounded components, each of which can be identified with an asymptotic tract.  
 As in section~\ref{dyn symmetries},   set $\omega_j=e^{\pi j/q}$,  $j=0, \ldots,2q-1$, the $q^{th}$ roots of $\pm 1$.  The labelling is such that if $j$ is even, $\omega_j$ is a root of $+1$ and if $j$ is odd, it is a root of $-1$.  The zeros of $f$ are of the form $z_{k,j}=|(k\pi)^{1/q}| \omega_j$, $j =0, \ldots, 2q-1$. 
 
 We work first with the basin of zero, $B$.  Consider $f^{-1}(B_0)=\cup_{k,j} U_{k,j}$ where $k \in \NN, j=0, \ldots 2q-1$.  Note that if $k=0$,  $U_{0,j}=B_0$ for all $j$.  For  $k \neq 0$, the $U_{k,j}$ are mutually disjoint. 
  Because  $B_0$ does not contain the asymptotic value, the $U_{k,j}$ are bounded. In addition, since $-B_0=B_0$ and $f(z\omega_j)=f(z)$, we see that $U_{k,j}=\omega_j U_{k,0}$, so their diameters are independent of $j$.   It therefore suffices to consider only those components $U_k=U_{k,0}$ along  the ray $\omega_0$.   Let $U_0=B_0$ and let $M_k$ be the diameter of $U_{k}$.   
   
    As we saw above, the map $f=\la \tan^p z^q$ is a composition of maps $f=L \circ P \circ T \circ Q$ where $L(z)=\la z, \, P(z)=z^p, \, T(z)=\tan(z)$, and $Q(z)=z^q$;  note that if $q=1$,  $Q(z)$ is the identity.

 First set $\zeta=Q(z)$ and write $F(\zeta)=f(z)$.  The Fatou set of $F$ is the image under $Q$ of the Fatou set of $f$.  Note that $F(\zeta +\pi) =F(\zeta)$,     $Q(z_k)=Q(z_{k,0})=k\pi$.  Let $V_k=Q(U_k)$.  By the periodicity of $F$, $V_{k+1}=V_k + \pi$, so these sets all have the same diameter, say   $\hat{M}$.  We also have for each $k$, $V_k=T^{-1}\circ P^{-1} \circ L^{-1} (B_0)$ for some inverse branches of $P$ and $T$.    Therefore $V_k$ is contained in a vertical strip  of width $\pi$, 
  $$H_j=\{z=x+yi \ | \ a_{j-1}\leq x\leq a_j, \text{ and } a_j-a_{j-1}=\pi\}.$$

 The map $Q$ sends the asymptotic tracts of $f$ to those of $F$, and these are the upper and lower half planes.  Therefore, there is some $y_0$ such that  the half planes $\Im \zeta >y_0, \Im \zeta < -y_0$ are in the asymptotic tracts of $F$.   Since we have assumed the asymptotic values tend to an attracting cycle, these asymptotic tracts belong to unbounded components of the Fatou set.  It follows that the components $V_k$ also lie between the horizontal lines $y=\pm y_0$.   
  Thus each $V_k$ is contained in a rectangle $R_k$ of height $2y_0$ and width $\pi$ so that $\widehat{M} \le 2y_0 \pi$.

Suppose $q>1$. 
Let $\calp_k=Q^{-1}(R_k)$ where we choose the branch of the inverse so that $U_{k} \subset \calp_k$.    Then using the derivative we estimate

\begin{equation}\label{pullbackest}
\calr_k  \subset \calr_k \approx  R_k/(k\pi)^{\frac{q-1}{q}}.
\end{equation}

Therefore,   only finitely many $U_{k}$ have diameter greater than $\epsilon$ and by symmetry, only finitely many of the other $U_{k,j}$ have diameter greater than $\epsilon$

  The remaining  bounded Fatou sets are contained in $Q^{-1}(R_k \setminus V_k)$.  These include both sets in $B$ and, if there is a non-zero attracting cycle, the bounded components of its basin.  
  Since  none of these bounded sets contains a critical value, the inverse  branches  defined on them yielding bounded preimages  are conformal homeomorphisms.  The sizes and shapes of these inverse images is controlled by Koebe estimates involving uniform bounds on the derivative of $F$ in $R_k$ and the diameters of the $U_k$. The diameters of these components shrink by     factors  of the order $1/|k\pi|^{n(q-1)}$  as $k,n$ go to infinity.

If $q=1$,  the function $f$ is the same as the function $F$.   Although the sets in each $R_k$ are  all the same size in the Euclidean metric,  as $k \to \infty$ they shrink in the spherical metric. As above,  for each $k$ the $n^{th}$ bounded preimages of the components of $B$ and those of the basin(s) of the non-zero cycle(s)  have diameters that shrink to zero so the lemma holds in this case as well.

\end{proof}

The last lemma in \cite{M2} is purely topological and applies here without modification.
\begin{lemma}
If $X$ is a compact subset of the Riemann sphere such that any component of its complement has locally connected boundary, and such that, for any $\epsilon>0$, at most finitely many of these complementary components have spherical diameter greater than $\epsilon$, then $X$ is locally connected.
\end{lemma}

Theorem~\ref{J is loc conn} now follows directly from these lemmas. 

 In the proof of lemma~\ref{key1}, working with hyperbolic maps, we defined the dynamic rays in $B_0$.   These rays are defined in the same way even if $f$ is not hyperbolic and their preimages define rays in the remaining components of $B$.  Similarly, if the map has a non-zero attracting cycle, there are dynamic rays in the periodic components of the cycle and their preimages are rays in the preimages of the components.  We use the notation $R_t$ for any of these dynamic rays. 
 
  If $t$ is rational,  the rays  $R_{t} $   are called {\em rational rays}.   It is clear that they are eventually periodic or periodic under iteration by $f$.   A direct corollary of Theorem~\ref{J is loc conn} is
 
 \begin{cor}\label{reppts}The landing points of the rational rays in $B_0$ are  eventually repelling periodic points. 
 \end{cor}

\subsection{Cantor Julia sets}
\label{sec:Cantor}

It follows from theorem~\ref{J is loc conn} that  if $f$ is hyperbolic and $B_0$ is not completely invariant,  the Julia set is locally connected.  In this section we show that if $B_0$ is  completely invariant, the Julia set is totally disconnected and homeomorphic to a Cantor set.

 \begin{thm}\label{Cantor}
 If the asymptotic values are in $B_0$, the Julia  set $J$  is a Cantor set and the action of $f $ on $J$ is conjugate to the one sided shift on a countable countable set of symbols.
 \end{thm}

\begin{proof} By theorem~\ref{Bottchercoor}, there exists a neighborhood $U$ of $0$ and a  map $\varphi$   conjugating $f$ to $z\mapsto z^{pq}$. Take  $r>0$ small enough so  that $U_0=\varphi^{-1}(D_r(0))$ does not contain the asymptotic values and its boundary  does not contain any point in the forward orbit of asymptotic values.  Let $U_n$ the component of $f^{-n}(U_0)$ containing $0$ so that  $U_n\subset U_{n+1}$ and $B_0=\cup_{n=0}^\infty U_n$. 

Since we are assuming  the asymptotic values are in $B_0$, by   theorem \ref{dichotomy}, $B_0=B$.
    We claim that $J=\widehat{\mathbb{C}}\setminus B$ is a Cantor set.
  
There exists an $N$ such that $U_{N-1}$ does not contain any asymptotic values but $U_N$ contains the asymptotic value; in fact, by construction it contains also contains $v'$. By lemma \ref{preimageSC}, $U_1, \cdots, U_{N}$ is a nested set of simply connected domains and $f^{-1}(U_N)$ is an infinitely connected  set whose complement is a collection of topological disks, each punctured at a pole.   We add the poles to each of these disks and enumerate them.  To do so, recall that they lie along the rays separating the asymptotic tracts and were labelled  in section~\ref{dyn symmetries} by:
\[  p_{k^j}=(p_k)^{1/q}\omega_j, \,  j \mbox{ even and } p_{-k^j}=(p_k)^{1/q}\omega_j,  \, j \mbox{ odd }\, k \in \NN, j=0, \ldots, 2q-1,\] 
where $p_k= |(k\pi+\pi/2)^{1/q}|$.  We therefore let $A_{\pm k^j}$ be the complementary set containing $p_{ \pm k^j }$ and set $W = \widehat\CC \setminus U_N$.   The maps  $f: A_{k^j} \to W $ and $f: A_{- k^j} \to W $  are branched coverings of degree $p$ over $\infty$. 

Note that $J\subset\cup_{k\in \mathbb{N}} (\cup_{j=1}^{2q-1} A_{k^j} \cup \cup_{i=1}^{2q-1} A_{ -k^j} \cup \{\infty\} )$ so that  $J$ has infinite  connectivity. It is easy to see that 
\[  F=\cup_{n\geq 1}f^{-n}(U_N) \text{ and } J=\cap_{n\geq 1} f^{-n}(\cup_{k\in \mathbb{N}} (\cup_{i=1}^{2q-1} A_{k^j} \cup_{i=1}^{2q-1} A_{-k^j}) \cup \{\infty \}).\]

The next step is to look at the preimages of the  $A_{\pm k^j}$;  recall that they contain poles or critical points but no  singular values and so any preimage of  $A_{\pm k^j}$ 

  maps injectively onto it.    For the sake of readability, we use the notation $k^j$ where we assume $k$ is positive if $j$ is even and $k$ is negative if $j$ is odd. Consider $A_{k_0^{j_0}}$ for a fixed $j_0,k_0$;   there are 
  $p$ components of $f^{-1}(A_{ k_0^{j_0}})$ in  each $A_{k^j}$;   we denote them  by $A_{k^j_l k_0^{j_0}}$,  where $l=1, \ldots, p$.  Thus, if  $x\in A_{k^j_l k_0^{j_0}}$, then $f(x) \in A_{k_0^{j_0}}$. 

We continue to take preimages and again the maps are injective.  To enumerate them carefully would require using three new indices for each successive preimage step.  Therefore, to keep the notation readable, at the $n^{th}$ step, we will write
$k^j_l \cdots k^j_l k^j$ where there are $n$ entries and the $k,j,l$ of each entry vary independently.      

We claim that  $\cap A_{k^j_l\cdots k^j}$ is a single point which, in turn, shows that $J$ is a Cantor set.  The proof of this claim is essentially the same as the proof of lemma~\ref{key2}.  In that argument we showed that the components of the basin of zero were contained in the pullbacks of rectangles of fixed size and shape.   The same argument shows that each of the sets $(A_{\pm k^j})^q$ is contained in a    rectangle of fixed height $2y_0$ and width $\pi$ containing  $p_k$.  Taking $q^{th}$ roots we get an estimate on the diameters of sets containing the $A_{\pm k^j}$.   Again, as in lemma~\ref{key2}, because the inverse branches of $f$ on these sets are conformal,  the successive preimages shrink by a definite factor;  this proves the claim.

Each set $A_{k^j_l\cdots k^j}$ contains a unique prepole  whose order is the number of entries.  Thus we can label it $p_{k^j_l\cdots k^j}$,  assign it the symbol $k^j_l\cdots k^j$ and note that $f(p_{k^j_l\cdots k^j})$ is the prepole whose symbol is obtained by dropping the first $k^j_l$.   Recall that the Julia set is the closure of the set of prepoles.   Putting the standard sequence topology on the set of symbols of arbitrary finite length, we can form its closure $\Sigma$ by adding the symbols $\cdots k^j_l k^j$ of infinite length and the shift map $\sigma$ is continuous in this topology.  \footnote{See \cite{DK} for a detailed construction of $\Sigma$.} Thus the map $\Xi$ that assigns the  prepole to its finite symbol extends to a continuous map 
$\Xi: \Sigma \to J$  such that $\Xi \circ \sigma = f \circ \Xi$.  
\end{proof}

 \section{Parameter Plane}
 \label{param}
 As we saw above, the dynamics of the functions $f_{\la} =\la  \tan^p z^q \in \calf$, $\la \in \CC^* =\CC \setminus \{0\}$,  depend on the location of the asymptotic values.     The functions in this family are what we called {\em Generalized Nevanlinna Functions} in \cite{CK1}.  Nevanlinna functions are meromorphic functions with $m$ asymptotic values and no critical values.   A neighborhood of infinity consists of $m$ distinct asymptotic tracts contained in sectors of angle $2\pi/m$.  The poles approach infinity along the rays defining the sectors.  These rays are called Julia directions. 
  By a classical theorem of Nevanlinna \cite{Nev1} this topological description of their mapping properties is characterized by an analytic condition:  they satisfy the relation $Sg(z)=P(z)$ where $Sg= (g''/g')' -(1/2)(g''/g')^2$ is the Schwarzian differential equation and $P(z)$ is a polynomial of degree $m-2$.   This means that if $h$ is topologically conjugate to a solution $g$ of this equation, and it is meromorphic, then $h$ also satisfies a Schwarzian equation with a different polynomial of the same degree.   
 
 A {\em Generalized Nevanlinna function} has the form $h=P \circ g \circ Q$ where $g$ is a Nevanlinna function and $P$ and $Q$ are polynomials of degrees $p$ and $q$ respectively.   Thus if $g$ has $m$ asymptotic tracts, $h$ has $mq$ tracts.  
 The following theorem was proved in \cite{FK}.
 
 \begin{thm}\label{top fam} If $h_2$ is a meromorphic function topologically conjugate to a generalized Nevanlinna function  $h_1=P_1 \circ g_1 \circ Q_1$   then $h_2$ is a generalized Nevanlinna function 
 $h_2=P_2 \circ g_2 \circ Q_2$ where the degrees of $P_i$ and $Q_i$ agree and the $g_i$ have the same number of asymptotic values.  
 \end{thm}
 
 As a corollary we have
 \begin{cor} Let $f_{\la} \in \calf$ be a hyperbolic function and let $h$ be a meromorphic function topologically conjugate to $f_{\la}$.  Then, up to conjugation by an affine transformation, $h=f_{\la'}$ for some $\la'$.
 \end{cor}
 
Another corollary that follows by standard quasiconformal surgery techniques, (for example\cite{BF}) is
 \begin{cor} If $f_{\la}$ is a hyperbolic function and $g$ is a meromorphic function topologically conjugate to $f_{\la}$, then $g$ is quasiconformally conjugate to $f_{\la}$.  \end{cor}
 
 It follows that the $\la$ plane has a structure that reflects the dynamics; that is, it consists of components containing topologically   quasiconformally  conjugate hyperbolic maps for which the orbits of the asymptotic values tend to a the superattracting fixed point $0$,  or tend to a non-zero attracting cycle.  These components are separated by a {\em Bifurcation locus} consisting  of parameters for which the dynamics are rigid.  
 
 \begin{figure}[htb!]
\begin{center}
\includegraphics[height=3in]{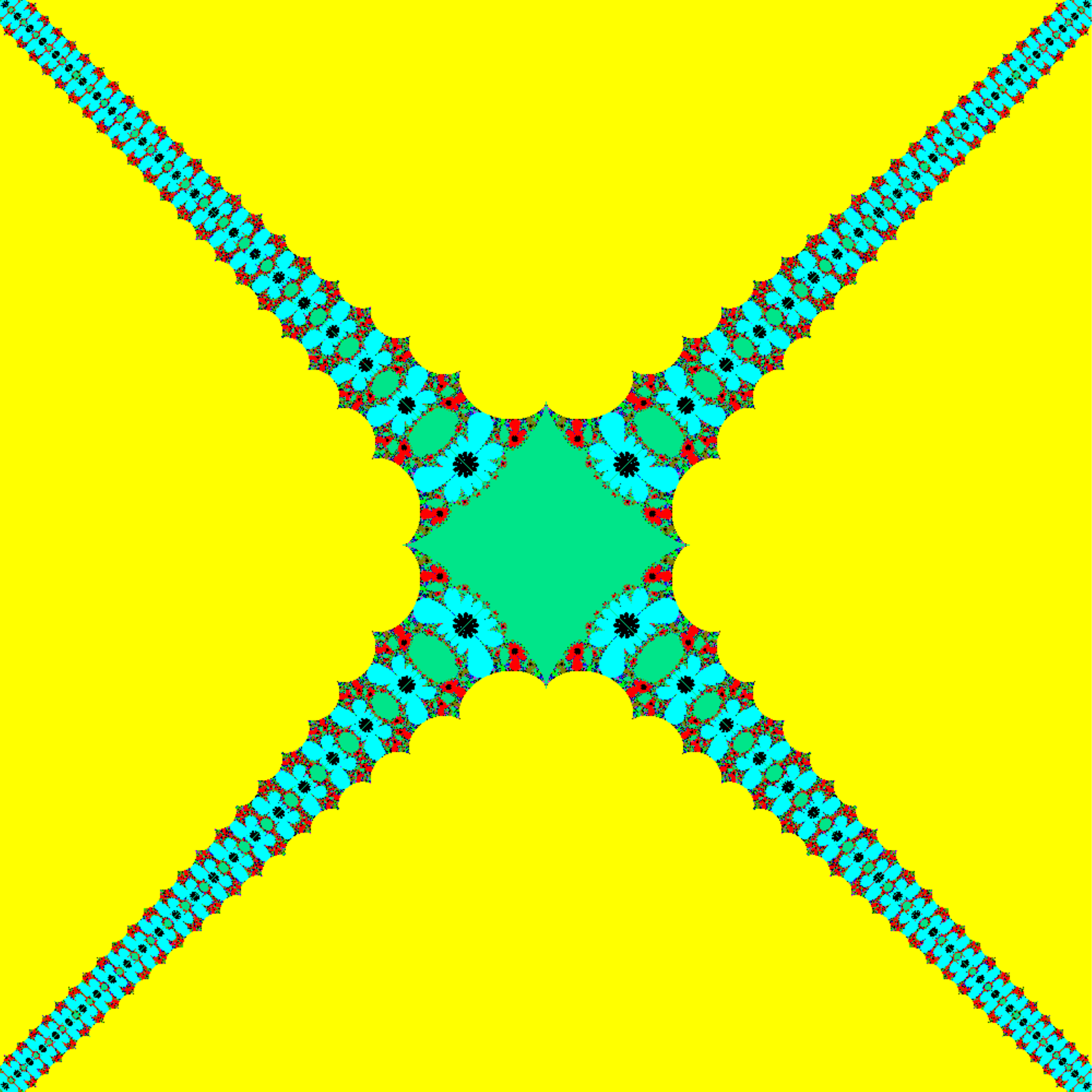}

\end{center}
\caption{ The parameter plane for $\la \tanh^3 z^2$. The green regions are capture components.  The yellow regions are period $1$ shell components; the cyan regions are period $2$; the red regions period $3$.}
\label{param32}
\end{figure}
 
 This gives us a natural division of the hyperbolic components:\\
 The hyperbolic components are divided into two categories:
\begin{itemize}
\item {\em Shell components} in which the asymptotic value  is attracted to a periodic cycle different from the origin.  By symmetry, if there are two asymptotic values, $\pm v_{\lambda}$,  either both are attracted to
the same cycle, or each is attracted to a cycle and the cycles are symmetric. We denote the collection of shell components by $\cals$ and the individual components for which $f_{\la}$ has an attracting cycle of period $n$ by $\Omega_n$.  Note that since $0$ is the only critical value, if  $\la \in \Omega_n$ the multiplier $m_{\la}$ of the  non-zero attracting cycle of $f_\la$,  is not zero.

\item {\em Capture components} in which the asymptotic value is attracted to $0$.  As we saw in section~\ref{basin of zero}, if there are two asymptotic values,  neither or both are attracted to $0$. 
We denote the set of capture components by $\mathcal C$ and divide this set into subsets $\calc_{\bf n}$, where
    \[ \mathcal{C}_\mathbf{n}=\{\lambda\ | \ n\geq 0 \text{ is the minimal } k \text{ such that } f^k_\lambda(v_\la)\in B_0 \}.  \] 
     
  Define the set of {\em centers} of  components of $\calc_n$, $n>0$, as 
\[ \mathcal{Z}_n=\{ \la\ |  \ n \geq 1 \text{ is the minimal integer such that } f^k_\lambda(v_{\la})=0\}. \] 
It is clear that these analytic equations have solutions; for example, if $n=1$ the centers are the points $|(n \pi)^{1/q}| \omega_j$.   Hence that the components of $\calc_n$ are not empty.    
    
      There is only one component $\calc_0$ and we call it the {\em central component}. It has no center because $\la=0$ is a parameter singularity.
\end{itemize}

\begin{figure}[htb!]
\begin{center}
\includegraphics[height=3in]{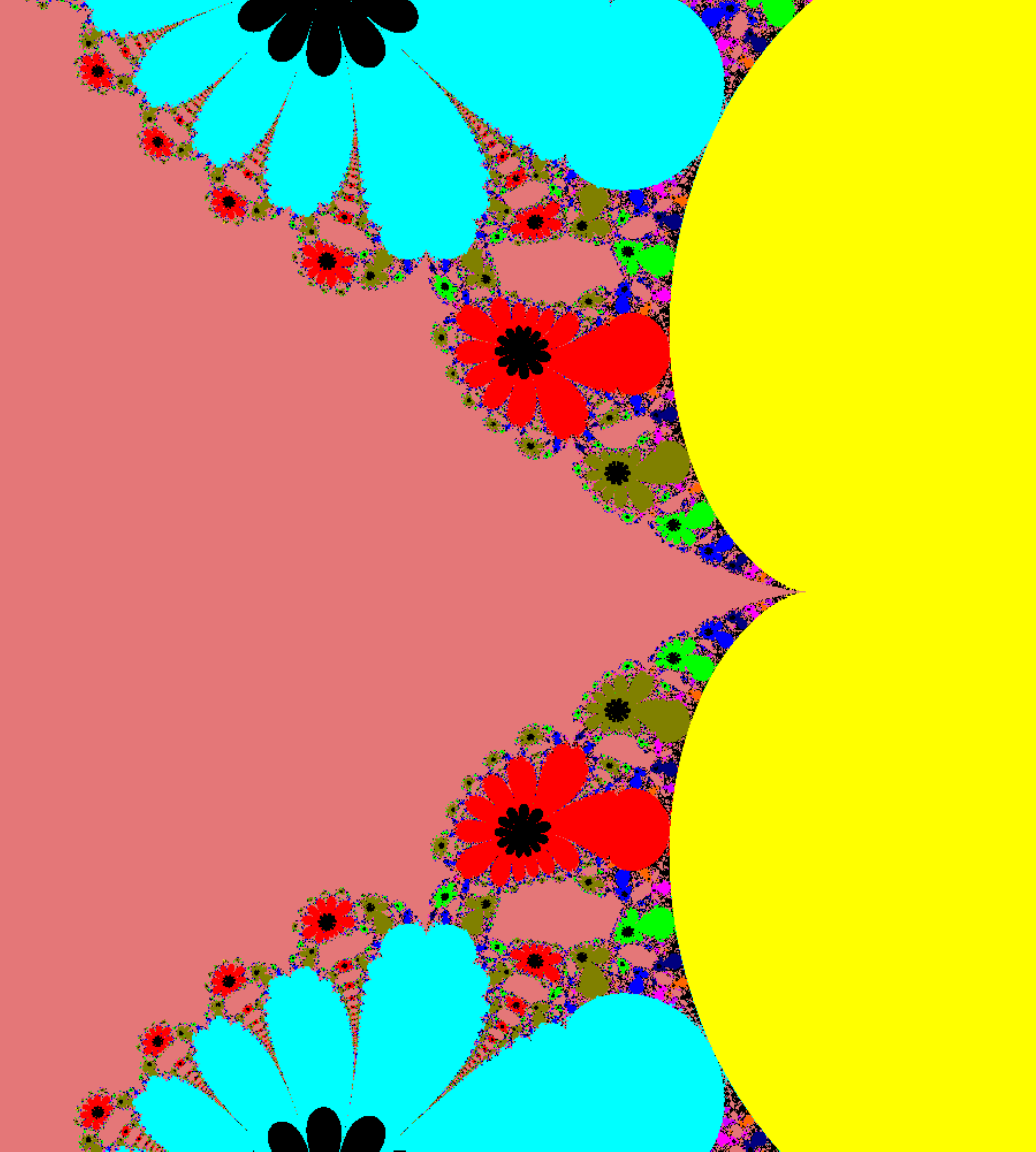}
\end{center}
\caption{A blowup of the central capture component near the cusp.  The capture components are dark pink. The shell components are colored as in figure~\ref{param32}.}
\label{central blowup}
\end{figure}

In section~\ref{sec:Cantor} we saw that the functions in the central capture component are precisely those whose Julia set is not connected, and is in fact a Cantor set.     Below we will see that this makes it quite different from the non-central capture components.  

Before we give a more detailed analysis of the hyperbolic components we discuss some general properties of the parameter plane. 

 \subsection{Symmetries in the parameter plane}\label{param symmetries}
 In section~\ref{dyn symmetries} we saw that the functions in $\calf$ exhibit a $2q$-fold symmetry if $p$ is even and a $q$-fold symmetry if $p$ is odd.   
 
The parameter plane also has symmetries and they also need to be described in terms of the parities of $p$ and $q$.

 Note first that we always have
 \begin{itemize}
\item    $ f_{-\la}(z) = -f_{\la}(z)$ so that $v_{-\la}=-v_{\la}$ and 
\item  $f_{\bar\la}(z) = \overline{f_{\la}(\bar z)}$, so that $v_{\bar\la}=\overline{v_{\la}}$ if $p$ is even and $v_{\bar\la}=\overline{v_{\la}'}$ otherwise.
\end{itemize}

These relations are reflected in the parameter plane, but exactly how depends on the parities of $p$ and $q$.  
   
  If $\la$ is real,  $f_{\la}$ leaves the real axis invariant. Also, since  $ \tan iz = i \tanh z$,  if $pq$ is odd $f_{\la}$ leaves the imaginary axis invariant and otherwise maps it to the real axis.    This is an example of a more general phenomenon relating lines in the parameter and dynamic planes: the rays $\omega_j t$, $t \in (0, \infty)$ divide the parameter plane into $2q$ sectors.   For $\la$ along these rays,   $f_{\la}$ exhibits invariance along the corresponding rays of the dynamic plane.  This is made explicit in the following. 
 
   \begin{prop}\label{syminv}
Suppose $t$ and $x$ are real, then
\begin{itemize}
\item For $j$ even,  $f_{t}(\omega_j x)$ is real and $f_{\omega_j t}$ leaves the line through $\omega_j x$ invariant.
 \item  For $j$ odd;   if $p$ is even,  $f_{\omega_j t}$ leaves the line through $\omega_j x$ invariant whereas if $p$ is odd, it maps the line through $\omega_j x$  to the perpendicular line through $i \omega_j x$.
 \end{itemize}
\end{prop}

The proof of this proposition follows from the next two lemmas.
Let $t>0$.  

 \begin{lemma}\label{conj1} Suppose $\lambda=t\omega_j$. 
\begin{enumerate}
 \item If $pq$ is even, then  $f_{\lambda}$ is conjugate to $f_t(z)=t\tan^p z^q$,
  \item whereas if $pq$ is odd,  then $f_{\la}$ is conjugate to 
  $g_t(z)=t\tanh^p(z^q)$ or to $g_{-t}(z)=-t \tanh^p(z^q)$.

 \end{enumerate}
 \end{lemma}

\begin{proof}
\begin{enumerate}
\item  if $p$ is even, $f_\lambda$ is conjugate to $f_t$ since
 $$f_\lambda(\omega_jz)=\omega_j t\tan^pz^q=\omega_jf_t(z).$$

  If $p$ is odd and $q$ is even,  then for $j$ even 
\[  f_{\omega_j t}(\omega_j z)=\omega_j t \tan^p(\omega_j z)^q=\omega_j t\tan^p(z^q)  \] and for $j$ odd,
\[  f_{\omega_j t}(-\omega_j z)=\omega_j t \tan^p(\omega_j z)^q= -\omega_j t\tan^p(z^q).  \]

\item  if $pq$ is odd 

\[ f_\la(i^p \omega_j z)=\omega_j t \tan^p (i^{pq}z^q)= \omega_j t \tan^p (\pm iz^q)=\pm i^p \omega_j t \tanh^p(z^q). \]
\end{enumerate}
\end{proof}

Denote the $q^{th}$ roots of $\pm i$ by $\xi_j = e^{(2j+1)\pi/2q}, j=0, \ldots 2q-1$.  For $j$ even, they are roots of $i$ and for $j$ odd, they are roots of $-i$.  
 \begin{lemma}\label{conj2}
Suppose $pq$ is even. If  $\la=\xi_j t$, then
$f_\la$ is conjugate to $g_t$.
\end{lemma}
\begin{proof}
Set $\eta=(i)^p \xi_j$. Then $\eta^q=(i)^{pq}\xi_j^q$

If $p$ is even, then $( \eta z)^q=\pm iz^q$  and \[ f_{\xi_j t}(\eta z)=\xi_j \tan^p(\eta z)^q=\xi_j t\tan^p(\pm iz^q)=(i)^p  \xi_j t\tanh^pz^q=\eta g_t(z).\]

Next suppose $p$ is odd, and $q$ is even. Suppose, furthermore, that $(i)^{pq}=1$. If $j$ is even, so that  $\xi_j^q=i$, then
\[ f_{\xi_j t}(\eta z)=\xi_j t\tan^p(\eta z)^q=\xi_j t\tan^p( iz^q)=(i)^p  \xi_j t\tanh^pz^q=\eta g_t(z).\]
If $j$ is odd, so that  $\xi_j^q=-i$, then
\[ f_{\xi_j t}(-\eta z)=\xi_j \tan^p(-\eta z)^q=\xi_j t\tan^p(- iz^q)=(i)^p  \xi_j t\tanh^pz^q=-\eta g_t(z).\]

If $(i)^{pq}=-1$, a similar calculation shows that  $f_{\xi_jt}$ is conjugate to $g_t$.

\end{proof}

{\em An example}:  See figures~\ref{param32} and~\ref{central blowup}. Consider $g_t(z)=t\tanh^3z^2$. We know that $\xi_0=e^{i\pi/4}$ is a square root of $i$. We will show that the map $f_{\xi_0 t}(z)=\xi_0 t \tan^3 z^2$ is conformally conjugate to $g_t$.

The asymptotic values of this map are $(\pm i)^3 \xi_0 t $. With the notation used above,  $\eta=i^3\xi_0$. Then the map $f_{\xi_0 t}$ is conjugate to $g_t(z)$ by $h(z)=-\eta z$ since 
\[ f_{\xi_0 t} (-\eta z)=\xi_0 t \tan^3((-\eta z)^2)=\xi_0 t (-i)^3\tanh^2(z^2)=-\eta t\tanh^3z^2.  \]

\begin{remark}

\begin{enumerate}
\item It is easy to check that the  iterates of the asymptotic values of $f_t$ or $g_t$ are all  on the real line. Moreover as in \cite{KK1} and \cite{CK1}, if $pq$ is odd, the dynamics of $g_t=t\tanh^pz^q$ and $g_{-t}=-t\tanh^pz^q$ are related:
If $g_t$ has two $n-$cycles, then $g_{-t}$ has two $n$-cycles or one $2n$-cycle depending on whether  $n$ is even or odd. Therefore, below, we will concentrate on the real maps $f_t(x)=t\tan^px^q$ and $g_t(x)=t\tanh^px^q$, where $x\geq 0$.
\item  If $pq$ is odd and $\lambda=\tau\xi_j$, $j=0, \ldots 2q-1$, then  $\la$ and the asymptotic values are on perpendicular lines and the iterates of the asymptotic values alternate between these lines.   In fact $f_\la$ goes through an interesting bifurcation process, see \cite{KK1} and \cite{CJK1} for a discussion of this when $f_t(z)=it\tan z$.
    \end{enumerate}
    \end{remark}

\begin{remark}\label{sector}  If follows directly from the above discussion that  to understand the properties of the parameter plane, it suffices to study their properties in a sector of angle width $\pi/q$.  Below, we will use either the sector $S_r$, bounded by the positive real ray through $\omega_0$ and the ray through $\omega_1$ or the sector $S_i$ bounded by the rays through  $\xi_{0}$ and $\xi_1$.
\end{remark}

\subsection{Shell Components}
\medskip
Properties of the shell components were studied for a more general class of families of meromorphic functions in \cite{FK}.   To give a more complete picture of the parameter plane for our families, and in particular to understand the deployment of the capture components,  we summarize the relevant results in \cite{FK}.  Recall that if $f$ is hyperbolic and  has a non-zero attracting cycle of period $n$, then $\la$ belongs to a hyperbolic component of the parameter space which we denote by $\Omega_n$.  We assume here for readability that there is only one asymptotic value.

\begin{thm}\cite{FK}\label{1FK}
\begin{enumerate}
\item The multiplier map $\rho: \Omega_n \rightarrow \CC^*$ is a universal covering map, so in particular, $\Omega_n$ is simply connected.   It extends continuously to $\partial\Omega_n$.  The boundary is piecewise analytic and the points on $\partial \Omega_n$ such that $m_{\la}=1$ are {\em cusps.
 \item At points of  $\partial\Omega_n$ where $m_{\la}=e^{2\pi i t/s}$, $t/s \in \QQ$, there is a standard bifurcation of the periodic cycle and the point is a common boundary point of shell components of period $s$ and $ns$.}
\item There is a unique point $\la^* \in \partial \Omega_n$ such that $f_{\la^*}^{n-1}(v_{\la^*})=\infty$; that is, the asymptotic value is a prepole of $f_{\la^*}$ and is part of a ``virtual cycle''.  That is,  if $\gamma(t)$ is asymptotic path for $v_{\la^*}$ in the dynamic plane of $f_{\la^*}$, then
\[ \lim_{t \to 1} f_{\la^*}^n(\gamma(t))=  \infty. \]

 Such a  $\la^*$ is called a {\em virtual cycle parameter of order $n$.} 
 \item 
 If $\la_n  \to \la^*$, $\la_n \in \Omega_n$, and $\la^*$ the virtual cycle parameter, then $\rho(\la_n) \to 0$; thus, $\la^*$ is called a {\em virtual center} of $\Omega_n$.
  \item Every virtual center  is a virtual cycle parameter.
\item The virtual centers of the $\Omega_n$ are accumulation points of virtual centers of components $\Omega_{n+1}$.
\end{enumerate}
\end{thm}

For the family $\calf$ we also have the following results proved in \cite{CK1}, 
\begin{thm}\cite{CK1}\label{allvirtcents}  If $\la^*$  is a virtual cycle parameter of order $n$ then it is on the boundary of a shell component $\Omega_n$, $n>1$ and is its virtual center.  
\end{thm}

\begin{thm}\cite{CK1}\label{allbdd} The only unbounded hyperbolic components in the $\la$ plane are the $2q$ shell components of order $1$.  Moreover, these are the only shell components of order $1$.  They are pairwise tangent to the lines $\omega_j$, $j=0, \ldots, 2q-1$.
\end{thm}

and 
\begin{thm}\label{virtcentsconv} Each virtual cycle parameter  of order $n>1$ is the virtual center of $2q$ shell components of order $n+1$.These shell components are pairwise tangent to vectors that divide a circle centered at the point into $2q$ equal parts. 
\end{thm}

These theorems show that the shell components are in many ways analogous to the components of the Mandelbrot set for polynomials. See figure~\ref{param32}.  

\medskip

 \section{Capture Components}
 \label{capture}
 The new results on parameter spaces in this paper concern the capture components.  Although the central and non-central components have some properties in common;  for example, by theorem~\ref{allbdd}, they are all bounded; there are important differences.    
 We divide the discussion between the central and non-central components.  We assume throughout that $\la$ is in the chosen sector.

 \subsection{Connectivity}
 
 We prove in this section that the capture components are simply connected.  We show it first for the central capture component and then for the non-central ones.  These results, together with Theorem~\ref{1FK} show that all the hyperbolic components are simply connected.

\subsubsection{The central component}
 Recall that by definition,  
 \[ \mathcal{C}_\mathbf{0}=\{\lambda \ |\  v_{\la}, v_{\la}'\in B_0\}. \] 
 
\begin{thm}\label{C0}
The set $\mathcal{C}_\mathbf{0}\cup \{0\}$ is connected and simply connected.
\end{thm}

\begin{proof} 
In theorem~\ref{Bott}, we defined the B\"ottcher coordinate $\phi_{\la}$ of $f_{\la}$ in terms of the B\"ottcher coordinate $\phi_h$ of the conjugated monic map $h_{\mu}$ where $\mu=\la^{1/(pq-1)}$ for some choice of the root.    It is easier to do the computations for using that coordinate.  We assume the roots are taken so that $\mu$ is also in the chosen sector. 

 By lemma~\ref{allin},  the pullbacks of each can be extended injectively until they  meet an asymptotic value: $v_{\la}$ for $f_{\la}$ and $w_{\mu}=i^p \mu^{pq}$ for $h_{\mu}$.  We can therefore define  maps 
 
\[ \varphi : \calc_0 \cup \{0\} \rightarrow \DD \mbox{ by } \la \mapsto \phi_{\la}(v_{\la}), \varphi(0)=0 \]  and 
\[ \varphi_{\mu} : \calc_{0, \mu} \cup \{0\} \rightarrow \DD \mbox{ by } \mu \mapsto \phi_{h}(w_{\mu}), \varphi_{\mu}(0)=0 \] 
where  $ \calc_{0, \mu}= \{ \mu = \la^{1/(pq-1)},  \, \la \in \calc_0.\}$.

To see that $\varphi$ is holomorphic with respect to $\la$,  it is enough to see that $\varphi_{\mu}$ is holomorphic in $\mu$.  
Note that $h_\mu(w_{\mu})=\mu^{pq}\tan^p(w_{\mu}^q)$ is holomorphic in $\mu$  and by lemma \ref{preimageSC}, its $n^{th}$ iterate, $h_{\mu}^n(w_{\mu})$ is also holomorphic.  It   never takes the values  $\sqrt[q]{k\pi}$ because if it did, $\la$ would be in a non-central component.  Thus 
\[  \frac{h_\mu^{n+1}(w_{\mu})}{( h_\mu^{n}(w_{\mu}))^{pq} }=\frac{\mu^{pq}\tan^p (h^n_\mu(w_{\mu})/\mu)^q}{(h^n_\lambda(w_{\mu})^{pq}} \]   is holomorphic and never equal to $0$.  Therefore

\begin{align}\label{Botch factor}
\big(\frac{h_\lambda^{n+1}(w_{\mu})}{( h_\lambda^{n}(w_{\mu}))^{pq} }\big)^{{pq}^{-(n+1)}}=1+\cdots
\end{align}  is holomorphic in $\mu$, and this implies that $\varphi(\mu_{\la})$ is holomorphic in $\mu$.

Second, the map $\varphi(\mu)$ is a proper map. Consider a point $\mu_0\in \partial \calc_{0,\mu}$;  the B\"ottcher map $\phi_{\mu_0}$ is a conformal map from the immediate basin of $h_{\mu_0}$ onto the unit disk. In particular, for any $\epsilon>0$, $\phi_{\mu_0}^{-1}$ is a single valued function on the disk of radius $1-\epsilon$. This property is preserved under any small perturbation of $\mu_0$, and so holds for any $\mu\in \calc_{0,\mu}$ sufficiently close to $\mu_0$ which implies that  $\varphi$ is a proper map from  $\calc_{0,\mu} \cup \{0\}$ onto $\mathbb{D}$.

By equation~(\ref{Botch factor}), in a neighborhood of $0$,
\[ \varphi(\mu)=w_{\mu}(1+\cdots) = i^p\mu^{pq}(1+ \cdots), \] thus $0$ is removable. That is, the map $\varphi$ can be extended as proper surjective map from $\mathcal{C}_\mathbf{0}^\mu \cup \{0\}\to \mathbb{D}$ of degree $pq$ branched only at $0$.
 
Since we are working in a sector, the roots are well defined  and  the map $\varphi$ is well defined on that sector.  We extend it to all of $\calc_0$ by the symmetry relations.  
\end{proof}

Just as in the dynamic setting, we use the map $\varphi$ to define parameter rays and gradient curves in $\calc_0$.

\begin{defn}\label{param rays} The inverse images in $\calc_0$ the rays in $\DD$, $\calr_\theta=\varphi^{-1}(s e^{2\pi i \theta})$, $s \in (0,1)$, and $\theta$ fixed in $ [0,1)$,  are rays in $\calc_{0}$.   If $\theta$ is rational, we say the ray $\mathcal R_\theta$ is rational. 
\end{defn}

The analogue of the rate of escape to infinity in the outside of the Mandelbrot set is $|\varphi(v)|$.   

A direct  corollary to theorem~\ref{Cantor} is 
\begin{cor}
 If $\lambda\in \mathcal{C}_{0}$, the Julia  set $J_{\lambda}$  is a Cantor set and the action of $f_{\lambda}$ on $J_{\lambda}$ is conjugate to the one sided shift on a countable alphabet. Moreover, if $\RR^+$ is the positive real axis the conjugacy is well defined on $\calc_{\mathbf{0}} \setminus \RR^+$.
\end{cor}
\begin{proof}  We saw in theorem~\ref{Cantor} that for any fixed $\la \in \calc_{0}$, $B_0$ is completely invariant and $J_{\la}$ is a Cantor set.   Since the asymptotic values are attracted to the origin, $f_{\la}$ is hyperbolic and  by standard arguments (see e.g. \cite{McM, KK1}) the Julia set moves holomorphically and injectively.  Because we removed $\RR^+$ from $\calc_0$,  $\la$ varies in a simply connected domain and so the map $\Xi$ is well defined in this domain. 
\end{proof}

\begin{remark} A component of the parameter space for which the Julia set is a Cantor set and the dynamics are conjugate to a shift map is often called the {\em shift locus};  therefore we will call $\calc_0$  the {\em shift locus for $\calf$.}  
\end{remark}

\medskip
Next we show that 
there is a disk inside the central component

\begin{lemma}\label{disk}  Let $t_*= (\pi/4)^{1/q}.$  The disk $D=D(0, t_*)$ is contained in the central capture component $\calc_0$. \end{lemma}

\begin{proof} 

First consider the function $h(x)=x/\tan x^q$ for real $x$. Its derivative $$h'(x)=\frac{\tan x^q-qx^q\sec^2 x^q}{\tan^2x^q}<0$$  for $x^q\in (0,\pi/2)$, since $$\frac{qx^q\sec^2x^q}{\tan x^q}=q\frac{x^q}{\sin x^q}\frac{1}{\cos x^q}> 1.$$ Thus for any $x\in (0,t_*)$, $h(x)=x/\tan x^q>h(t_*)=t_*$ and so
  \begin{equation} \label{eq1}
  t \tan x^q< x  \ \ \ \text{ for }x\in (0,t_*), t\in(0,t_*].
  \end{equation}

 Therefore, for any $t\in (0, t_*]$ and $x\in (0, t_*)$,  by inequality (\ref{eq1}), we have $$f_t(x)=t\tan^px^q<t\tan x^q<x.$$

Next, we use this inequality to prove that
\begin{equation}\label{eqn2} |f_{\la}(z) |=|\la | \tan^p(z)^q| < |z|  \end{equation}
for  all $\la \in D$ and all $z \in D$. This implies that for any such $\la$,  the whole disk $D$ is in $B_{0}$ and, in particular, that the asymptotic value $v_{\la}=i^p \la$ is attracted to the origin.

Note first that  if $z=x+iy\in D$, then  $|\tan z|
\leq |\tan x|<1$.  Now since $\pi/4 < t_* <1$, for $z, \la \in D$,  inequality (\ref{eq1}) implies
\[  |f_{\la}(z)| =|\la (\tan(z^q))^p| < |\la \tan z^q|< | \la \tan{\Re z^q}| <|z|. \]

which proves  inequality~(\ref{eqn2}). 
\end{proof}

It follows that the segment $(0, t_*) \in \calc_0$ is the ray $\calr_0$; its endpoint is clearly not in $\calc_0$ and so it is a boundary point of $\calc_0$.  Moreover, it is obvious that the ray lands there. 

 The next two theorems show there are many points on the boundary of the disk $D=D(0,t_*)$ that are also boundary points of $\calc_0$.  
 
 \begin{defn}\label{Misz} If $\la$ is a parameter such that the asymptotic value of $f_{\la}$  lands on a repelling fixed point, then $\la$  is called {\em Misiurewicz parameter}.   
 \end{defn}  
\begin{remark} Misiurewicz points were initially defined for the family of quadratic polynomials as parameter points where the critical value eventually lands on a repelling periodic cycle.  The critical point could not  belong to the cycle because then the cycle would contain a critical value and be attracting.  
 In the family $\calf$, if $pq=2$, the asymptotic values are omitted and so cannot belong to any periodic cycle; they may, however land on one, and if they do, the cycle must be repelling.     If, on the other hand,  $pq>2$,  the asymptotic values have infinitely many preimages;  if one of these belongs to a periodic cycle, so does the asymptotic value.     If this is the case, and if the cycle is repelling, we again call the parameter a Misiurewicz point. 
\end{remark}

Recall that $\omega_j$  $j=0, \ldots ,2q-1$ are the $q^{th}$ roots of $\pm 1$ respectively for $j$ even and odd, and $t_* =(\pi/4)^{1/q}.$

\begin{thm}\label{Mpts} If $pq$ is even, then any point the form $\lambda=t_{*} \omega_j$  is a boundary point $\mathcal{C}_0$  and is a  Misiurewicz  parameter.   There are $2q$ such points, equally distributed on the circle and one of them is on the positive real axis. 
  \end{thm}

\begin{proof}  By hypothesis, $f_{t_*}$ is even so 
$f_{t_*}(v_{t_*})=f_{t_*}(-t_*)=f_{t_*}(t_*)=t_*$.  
Also, 
\[ f'_{t_*}(t_*)=pq \, t_{*}^{q}\tan^{p-1}(t_*^q)\sec^2(t_*^q)=pq\pi/2>1\]
 and  $t_*$ is a repelling fixed point.  
 In other words, the image of the asymptotic value is a fixed point, and $\la=t_*$ is a Misiurewicz parameter.

 Since $pq$ is even,  Lemma \ref{conj1} applies, and $f_{t_*w_j}$ is conjugate to $f_{t_*}$;   these are also  Misiurewicz parameters. 
\end{proof}

Theorem~\ref{ubddcomps} implies
\begin{prop}\label{mptsbdy} If $\la^* = t_* \omega_j$ is a Misiurewicz parameter on the boundary of $\calc_0$, then $B_{\la^*}$, the basin of zero for $f_{\la^*}$, contains infinitely many unbounded simply connected components.
\end{prop}

\begin{proof} The proof will follow from theorem~\ref{ubddcomps} if we show that the asymptotic value $v_{\la^*}$ is an accessible boundary point of the immediate basin of zero for $f_{\la^*}$.  

Without loss of generality, assume $j=0$.   Since $t_*$ is real, the immediate basin  $B_0$  for $f_{t_*}$ contains the rays  $\phi_{t_*}^{-1}(i^p r)$, $r \in [0,1)$;  the limit as $r \to 1$ exists  so it is accessible.  \end{proof}

\medskip
There is another set of  $2q$ equally distributed rays in $\calc_0$ that extend outside $D$ and end in parabolic cusps.  

More precisely,

\begin{thm}~\label{cusps} Let $\omega_j$, $\xi_j$, $j=0, \ldots 2q-1$ be the $q^{th}$ roots of $\pm 1$ and $\pm i$ respectively.   
\begin{enumerate}
\item
If $pq$ is even,   the endpoints of the lines  $t \xi_j$, in $\calc_0$  are points for which the functions have a parabolic fixed point and hence a cusp.
\item If  $pq$ is odd, the endpoints of the lines  $t \omega_j$  in $\calc_0$  are points for which the functions have a parabolic period two cycle.  More precisely,  if $ pq \equiv 1 \mod 4$ and j is even, the function at the endpoint has two parabolic fixed points, whereas if j is odd, the it has a parabolic cycle with period two. If $pq \equiv 3 \mod 4$, the parity of $j$ is reversed.  
\end{enumerate}
\end{thm}

The proof follows by applying 
 
 the following lemma about a monotonic function of a real variable with an asymptotic value   to the family $f_t(z)=t\tanh^p z^q$.

\begin{lemma}\label{reallines}
Let $f_t(x)=f(t,x)$ be a family of analytic real maps satisfying the following:
\begin{itemize}
\item for each $t>0$, $f_t'(x)>0$ for all $x>0$,  and for each $x>0$, $\frac{df(t,x)}{dt}>0$ for all $t>0$.
\item for each $t$, there is a $b_t >0$ such that $f_t: [0, \infty) \rightarrow [0,b_t)$ and $b_t \to 0$ as $t \to 0$. 
\item $\displaystyle \lim_{t\to \infty} f_t(x)=\infty$
\item $\displaystyle \lim_{x\to\infty} f_t(x)=b_t$ and $\displaystyle \lim_{x\to 0} f_t(x)=0$
\end{itemize}  
Then there exists a $t_0$ such that 
\begin{enumerate}
\item For every fixed  $t$ in $(0,t_0)$, and for all $x>0$, $\lim_{n \to \infty} f^n_t(x)\to 0$.
\item $f_{t_0}$ has a positive parabolic fixed point.
\item If $t=t_0$, $f_t$ has an attracting fixed point other than $0$. 
\end{enumerate}
\end{lemma}

\begin{proof}
By hypothesis, for any fixed $x$,  $f(t,x)$ and its partial derivative with respect to $t$ are monotonic positive increasing functions of  $t$ and $f(t,x) \to \infty$ as $t\to \infty$;  thus  there exists a $t_1=t_1(x)$ such that $f_{t}(x)>x$ for all $t\geq t_1$.
Moreover, since for any fixed $t$, $\lim_{x\to\infty} f_t(x)=b_t >0$,   for large $x$, say $x >b_t$,  $b_t >f_t(x)$ and so $x> f_t(x)$. Therefore for any $t>t_1$, there exists an $x=x(t)>0$, such that $f_t(x)=x$ and $x(t)$ is a continuous function of $t$.

Define the set $$S=\{t \ |\ \text{there exists an } x=x(t) \in (0,\infty) \text{ such that } f_t(x)= x\}.$$  By the above, the set $S$ is non-empty.

Let $t_0$ be the infimum over all $t$ in the set $S$.
It is easy to see that when $t$ is very small, $f_t(x)<x$  for all $x>0$.  Thus for such $t$, $f_t$  has no positive fixed points so that $t_0>0$.  Moreover, for any $t\in (0,t_0)$, $f^n_t(x)\to 0$ as $n\to \infty$.

 Let $x_0 >0$ be a solution of  $f_{t_0}(x_0)=x_0$. We claim that $x_0$ is a parabolic fixed point with multiplier $+1$.

 Suppose not.   Since for each $t>0$, $f_t'(x)>0$ for all $x>0$, $f_{t_0}'(x_0)>0$.
  If $f_{t_0}'(x_0)>1$,  since $f_t'(x)$ depends continuously on $t$,   $f_t'(x)>1$ in some  interval $I$ of $t_0$, and thus $f_t(x)=x$ has a solution for all $t$ in this interval contradicting the minimality of $t_0$.   Similarly, if $f_{t_0}'(x_0)<1$,  then $f_t'(x)<1$ in some  interval $I$ of $t_0$ and $f_t(x)=x$ has no solution in $I$, again contradicting the minimality of $t_0$. Therefore $x_0$ is a parabolic fixed point.

For any $t > t_0$, $f_t(x_0)>x_0$ and $\lim_{x\to \infty} f_t(x)\to b_t$,  so $\lim_{x \to \infty} f_t'(x) =0$ and there is another solution of $f_t(x)=x$ with derivative less than $1$. That is, there exists an attracting fixed point $x\in (x_0,\infty)$.  

\end{proof}

As a corollary we have
\begin{cor}\label{parabbound}There exists a $t_0> t_*$, depending on $p, q$ such that
for  the family $f_t(z)=t\tanh^p z^q$
  \begin{enumerate}
  \item if $t\in(0, t_0)$, $f^n_t(x)\to 0$ for all $x\geq 0$.
\item if $t=t_0$, $f_t(x)$ has a parabolic fixed point.
\item if $t>t_0$, $f_t(x)$ has another attracting fixed point.
  \end{enumerate}
  \end{cor}

\begin{proof}
Note first that if $pq$ is even, $f_t(x)=t\tanh^px^q$ satisfies $f_t(-x)=f_{t}(x)$ and $f_{-t}(x)=-f_t(x)$ whereas if 
  $pq$ is odd,  $f_t(-x)=f_{-t}(x)$.    Thus if we confine our discussion to $x, t >0$,  the above lemma applies.  Therefore we need only 
    show that 
  that $t_0 > t_*=(\pi/4)^{1/q}$ for this family.  In fact, we will show $t_0 \geq 1 > t_*$.

 This will follow directly by showing
\[ t_* \tanh^p x^q <\tanh^p x^q < x,  \, \, \mbox{ for all } x >0.  \, \,  (**) \]
Note that $t_*<1$ and $\tanh^px^q <1$ so that $ \tanh^px^q <x$ for all $x \geq 1$. Thus  we need only prove $(**)$ for $x<1$.  In this case $x^q \leq x$ and $\tanh^px^q \leq \tanh x^q$ so that
\[ \frac{t_* \tanh^px^q}{x} < \frac{\tanh^px^q}{x} <  \frac{\tanh^px^q}{x^q}<\frac{\tanh x^q}{x^q}.\]

Therefore, we need only show that  $\tanh x<x$ for $x>0$.   This follows  because if
$g(x) = x - \tanh x$,  then $g(0)=0$ and $g'(x)=1-\sech^2 x >0$  so that $g(x)>0$ for all $x>0$.
\end{proof}

\begin{proof} [Proof of Theorem~\ref{cusps}.]
Lemmas~\ref{conj1},  \ref{conj2} and \ref{reallines}, together imply:
\begin{enumerate}
\item if $pq$ is even,  the line segment $t\xi_j, \ t\in [0,t_0)\in \mathcal{C}_0$, and  therefore $f_{t_0\omega_j}$ is conjugate to $g_{t_0}$ and  has a parabolic fixed point.

\item if $pq$ is odd, the line segment $t\omega_j, \ t\in [0,t_0)\in \mathcal{C}_0$, and therefore $f_{t_0\omega_j}$ is either conjugate to $g_{t_0}$, which has two parabolic fixed points, or to $g_{-t_0}$, which has a parabolic cycle of period $2$.

\end{enumerate}
\end{proof}

\subsection{Rational rays in $\calc_0$}
In the last section we considered the parameter rays $\omega_j t$ and $\xi_j t$ in $\calc_0$.  Their images under the map $\varphi$ defined in theorem~\ref{C0} using the B\"ottcher maps in the dynamic planes are rays in the disk $\DD$.  The theorems above show that if $\la$  belongs to one of these  rays in $\calc_0$, the asymptotic value, $v$ lies on a ray in the dynamic plane that is  invariant under $f$.
 Moreover, those rays are radii of a conformal disk in the parameter plane whose endpoints are either Misiurewicz points, points for which the function has a parabolic cycle.   This is true more generally.

\begin{thm}\label{ratl rays} The  rational rays $\calr_\theta \in \calc_0$ 
land at  parameters for which the asymptotic values  either belong to, or  land on repelling periodic cycles, or are attracted to parabolic cycles.
\end{thm}

\begin{proof} We saw above that the first case of the theorem is true for rays with $t=\omega_j$ and the second for rays with $t= \xi_j$.  
It will suffice therefore to consider rational rays $\theta$ in the sector between $t_0=\omega_0=0$ and $t_1=\omega_1$.

 Note that any accumulation point  $\la_{\infty}$ of $\calr_\theta$ in  $\partial \calc_0$ is finite by theorem~\ref{allbdd}.    The functions $f_{\la}$, their asymptotic values, prepoles and periodic points depend holomorphically on $\la$; the dependence is not, however, necessarily injective in a neighborhood of $\la_{\infty}$.  In addition, $\varphi$ is holomorphic in $\la$ so $\calr_{\theta}$ is real analytic in $\tau$.  

Assume now that $\theta$ is a fixed rational.  In the dynamic planes,  depending on $p,q$ and $\theta$,  but not on which point $\la$ on the ray $\calr_{\theta}$,  the asymptotic value $v_{\la}$, and its forward orbit,  may lie  on rays whose arguments differ by $\pi/2$.  For a given point on $\calr_{\theta}$,  we single out the dynamic ray  that contains  $v_{\la}$ and by abuse of our notation above, denote it by 
\[ R_{\theta}(s)=R_{\theta,\la(s)}=\phi_{\la(s)}^{-1}(se^{i\theta}), \, s \in (0,1).  \]

Since $\theta$ is rational,    there  are minimal integers $k,l$ such that $f_{\la(s)}^{k+l}(R_\theta) \subseteq f_{\la(s)}^l(R_\theta)$ and a $k$-cycle of rays, $\{ R_{\theta}^{l+j}=f_{\la(s)}^{j+l}(R_\theta)\}$, $j=0, \ldots,  k-1$, that is periodic.  

 Unless we need to emphasize $\la(s)$, for readability we will  suppress it.   Note that for  some value of $s$, $f^l(v ) = R^l_\theta(s)$. 
     If $l=0$, $R_{\theta}(s)=R^0_{\theta}(s)$ is periodic and the ray $R_\theta^{k-1}$ is  either an asymptotic path for all $f_{\la(s)}$ or it is not an asymptotic path for any  $f_{\la(s)}$.   Since $l$ is the minimal integer such that $f^l(v_{\la})$ is on a periodic ray,  if $l>0$,  none of the rays in the cycle is an asymptotic path.
     
 Below we assume $p$ is even so that there is only one asymptotic value.  The argument for $p$ odd is similar and left to the reader.     

 \begin{itemize}  
 \item
First suppose  that $l=0$, and   for all $s$, the  rays $R_{\theta}^{k-1}$ are asymptotic paths.  Using the maps $\phi_{\la(s)}$, we can conjugate the family of maps $f^k_{\la(s)}$ to a family of  real maps $g_\sigma:[0,\infty) \rightarrow [0,b_\sigma]$, $s=s(\sigma)$,   satisfying the hypotheses of  lemma~\ref{reallines}.    Then there exists a $\sigma_0$ such that if $\sigma \in (0,\sigma_0)$ the asymptotic value of $g_\sigma$ is attracted to $0$ and $g_{\sigma_0}$ has a parabolic fixed point.   Set $\la_{\infty}=s(\sigma_0)$;  it follows that $f_{\la_{\infty}}$ has a parabolic cycle and is in  $\partial \calc_0$.  It is an accumulation point of the ray $\calr_\theta$, and because the $f_{\la}$ with parabolic cycles  of order $k$ form a discrete set, it is unique and the ray lands there.

\item
Now suppose that either $l=0$ and for all $s$, the  rays $R_{\theta}^{k-1}$ are not asymptotic paths or that $l>0$.  It follows that for all $s$, none of the dynamic rays in the cycle, $R_{\theta}^{l+k}$ is an asymptotic path;  Therefore each of these periodic rays, including the one containing the asymptotic value if $l=0$, ends at a boundary point of the region of injectivity, $\calo_{\la(s)}$,  of the B\"ottcher map $\phi_{\la(s)}$. 
 Because the rays are periodic, their endpoints $b_{s,j}$, $j=0, \ldots, k-1$, form a periodic cycle   and the cycle and its multiplier are  holomorphic functions of $\la(s)$.   Since since $\la(s)$ is in $\calc_0$,  the cycle  is repelling.   The rays   contain all of the forward orbit of the asymptotic value.

Note that although, if $l=0$, one of the periodic rays contains an asymptotic value, it is not behaving locally like an asymptotic value because its preimage is not an asymptotic path. 

As $s \to 1$, the cycle persists and in the limit it is either parabolic or repelling.  Thus there is a unique limit point $\la_{\infty}=\lim_{s \to 1} \la(s)$. 

If the  cycle for $s=1$ is parabolic, the asymptotic value $v_{\infty}=v_{\la_{\infty}}$ of $f_{\la_{\infty}}$ is in the immediate basin of the cycle and so $l=0$ and it is attracted to the parabolic point of the limit cycle.  That is,  the boundary of the component of the parabolic basin containing $v_{\infty}$ shares the parabolic point with the boundary of $B_{0, \infty}$.

  Therefore if $l>0$, the limit cycle is repelling.
Now suppose the cycle for $s=1$ is repelling.     Then $v_{\infty}$ lies in the Julia set because it is no longer attracted to zero and there are no other parabolic or attracting cycles to attract it.  The subsequence of  $f^{-n}(v_{\infty})$  that lies on the limit of the periodic cycle of rays tends to the limit repelling cycle.  Therefore, $v_{\infty}$ eventually lands on this cycle and the point is a Misiurewicz point. 

Summarizing,  we have shown  that if $l>0$,  $\la_{\infty}$  is a Misiurewicz point and if $l=0$, it is either a Misiurewicz point or a parabolic point.   
 \end{itemize}

 \end{proof}

We remark that the above theorem implies that  the endpoint of a rational ray in $\calc_0$, cannot be a virtual cycle parameter.

\subsection{The components of $\mathcal{C}_\mathbf{n}$}

Recall that the non-central components are defined by 
  \[ \mathcal{C}_\mathbf{n}=\{\lambda\ | \ n\geq 0 \text{ is the minimal } k \text{ such that } f^k_\lambda(v_\la)\in B_0 \}.  \]

and the set of {\em centers} of $\calc_n$ by
\[ \mathcal{Z}_n=\{ \la_n \ n \geq 1 \text{ is the minimal integer such that } f^n_\lambda(v_{\la})=0\}. \]

 The map $\Xi$ of theorem~\ref{Cantor} that assigned a unique symbol of length $n$ to each prepole of order $n$ can be used to assign a similar symbol to the prezeroes and thus to the centers of the non-central components yielding an enumeration scheme  for these components.  
It is also clear from the definitions that 
\begin{prop}
If $\lambda\in \mathcal{Z}_n$, then $\lambda$ is in a component of   $\mathcal{C}_n$.
\end{prop}

Using standard techniques of quasiconformal surgery, see e.g. \cite{FG}, Prop. 4.5, the non-central components are simply connected and can be enumerated by the centers. Precisely, 

\begin{thm}\label{Cn}
Each component of $\mathcal{C}_\mathbf{n}$ is open, simply connected and contains a unique center. 
\end{thm}

\begin{proof}[Sketch]
For any component $C_n^i\in \mathcal{C}_\mathbf{n}$,   define the  map $\psi: C_n^i\to \mathbb{D}$ by 
\[ \psi(\lambda)= \phi_\lambda (f_\lambda^n(v))\] 
 where $\phi_\lambda: B_{0,\la} \rightarrow \DD$ is the B\"ottcher map for $f_{\la}$ defined on its immediate basin of $0$  in theorem~\ref{Bott}.   Since the only singular value in $B_{0, \la}$ is the origin, $\phi_{\la}$ is defined and injective on the whole basin.  
 Note that if $\la^*= \psi^{-1}(0)$, $\la^*$ is a center.   Moreover, 
 as in the proof of Theorem \ref{C0}, the map $\psi$ is a proper map.  

We next show that the map $\psi$ is a local homeomorphism.  To this end   choose 
  $\lambda \neq \la^* \in C_n^i$, and set $\xi=\psi(\lambda)\neq 0$.  
 For any $\xi \in \DD$,  standard surgery techniques  yield a quasiconformal map $\tau_{\xi}$ such that  $F_{\lambda(\xi)}=\tau_{\xi} \circ f_{\la^*} \tau_{\xi}^{-1}$.   Normalizing so that $\tau_{\xi}$ fixes $0$  and  the pole $(\pi/2)^{1/q}$, it follows from the generalized form of Nevanlinna's theorem, \cite{CK1}, that $F_{\la(\xi)}= f_{\la(\xi)} \in \calf$.   Moreover,  by the conjugacy, $ f_{\la(\xi)}^n(v_{\la(\xi)})$ is in the immediate basin of zero, $B_{0, \la(\xi)}$,  and  
 has B\"ottcher coordinate $\xi$.   Thus $\phi_{\la^*}^{-1}\circ  \psi: C_n^i \rightarrow B_{0, \la^*}$ is a homeomorphism, which proves the theorem. 
\end{proof}

We can use the map $\psi$ to define rays for each $C_n \in \calc_n$ as we did for $\calc_0$:  set $\calr_{t}^n= \psi^{-1}(se^{2\pi i t} )$, $s \in [0, 1)$.   As for $\calc_0$, if $t$ is rational, the ray is either periodic or preperiodic.   The proof of theorem~\ref{ratl rays} applies here to prove

\begin{thm}\label{ratl rays n} The rational rays $\calr_\theta \in C_n \subset\calc_n$ land at  parameters for which the asymptotic values    land on repelling periodic cycles and hence are Misiurewicz points. 
\end{thm}

\begin{proof}  The proof that the rational rays land is very similar to the proof of theorem~\ref{ratl rays} and we leave it to the reader to fill in the details.   To prove that the limit cycles are repelling and not parabolic, recall that we proved that there could only be a parabolic cycle in the limit if the ray containing the asymptotic value belongs to a periodic, and not a preperiodic ray for $s<1$.   Since $\la(s) \in C_n$, the asymptotic value $v_{\la(s)}$  is not in  $B_{0, \la(s)}$, but in $B_{n,\la(s)}=f^{-n}_{\la(s)}(B_{0,\la(s)})$ for some branch of the inverse.  This implies that the dynamic rays containing $v_{\la}(s)$ cannot be periodic.   Thus in the argument of theorem~\ref{ratl rays}, we are in the $l>0$ case; the limit cycle is preperiodic and not periodic.  It therefore cannot be parabolic proving the theorem.

  \end{proof}

 An immediate corollary is
\begin{cor} The landing points of the rational rays of the non-central capture components are not   accessible boundary points of   a shell component. 
\end{cor}
\begin{proof}  In \cite{FK} it is proved that  at all boundary points but one of a shell component $\Omega$,  $f$  has a neutral cycle; the exception is the virtual center where  it has a  virtual cycle.   Moreover, the boundary is piecewise analytic;  it is parameterized by the real line and has parabolic cycles at the rational numbers;  it  has cusps  at the integers.  If $\la_t$, $t$ rational,  were the landing point of a curve in   $\Omega$, $f_{\la_t}$ would have a parabolic cycle;  if it were also  a landing point of a rational ray $r_{t'}$ in a component of $\calc_n$, $f_{\la_t}$ would be a Misiurewicz point.  Thus, this cannot happen. 
\end{proof}

\section{Bifurcation locus}

Define the bifurcation locus as the complement of the hyperbolic components: 
$$\mathcal{B}=\mathbb{C}^*\setminus (\mathcal{S} \cup \mathcal{C})$$

\begin{thm}\label{bifloc1} Both the centers of the capture components and the Misiurewicz
points cluster on boundary of the capture components. \end{thm}

\begin{proof} If the set, $\calz$, of centers of capture components were not dense in the boundary of the capture components, $\partial\calc$, there would exist an open set $U$ which meets $\partial\calc$ but contains no solution of $f^n_\lambda(v)=0$ for any $n\geq 1$. Consider the family of maps   $h_n(\lambda)=f_\lambda^n(v)$ defined on $U$,
for all $n\geq 1$.  
Then $h_{n-1}(\la)$  omits the values $k\pi$, for every $k\in \mathbb{Z}$ because  if $f_{\la}^{n-1}(v)=  k \pi$, then $f_{\la}^n(v)=0$ and $\la \not\in U$.
 
Therefore $h_n(\lambda)$ is a normal family for it misses more than $3$ values.  Since, however, $U$ intersects $\partial \calc$, it contains an open set $V$ in  a capture component so that  a limit function of the $h_n$ is the constant function $0$ on $V$ and hence $U$.  It also contains   points $\lambda$ that are not in any capture component so that $h_n(\la) \not\to 0$; thus  the family is not normal.  This contradiction indicates that the centers cluster on the boundary of the capture components.

If the set of Misiurewicz points is not dense on the $\partial \calc$, there exists an open set $U$ which meets the boundary and contains no Misiurewicz points.  Again consider the family  
$h_n(\lambda)=f_\lambda^n(v)$ for all $n\geq 1$. 
 Now suppose that for some $n\geq 1$ and some $k$,  $h_n(\lambda)=(k\pi+\pi/4)^{1/q}$.  Then $$h_{n+1}(\lambda)=f_\lambda(h_n(\lambda))=f_\lambda((k\pi+\pi/4)^{1/q})=\lambda;$$
that is, $\lambda$ is a periodic point   in the forward orbit of $v$, so it    is a Misiurewicz point in $U$.
 Therefore the family $h_n$ omits the values $(k\pi+\pi/4)^{1/q}$ for all $k\in \mathbb{Z}$ and  $h_n(\lambda)$ is a normal family.  As above,  we have a contradiction since $U$ intersects $\partial \calc$, so the family cannot   be normal.  This implies that the Misiurewicz points are dense in $\partial \calc$ as claimed. 
\end{proof}

\begin{thm}\label{bifloc2} 
The centers and virtual centers accumulate on $\mathcal{B}$.  
\end{thm}
\begin{proof}  By theorem~\ref{1FK}, the virtual centers are boundary points of shell components and hence belong to $\calb$.  Moreover, each center is an accumlation point of centers of higher order. 

In \cite{CK1} we proved that the capture components and shell components of period greater than one for $\calf$ are contained between curves that are asymptotic to the Julia directions.  We also proved that there are curves joining the virtual centers at the points $((k+1/2)\pi)^{1/q}$ on the Julia rays to the period one components that separate the non-central capture components centered at the zeros $(k\pi)^{1/q}$ on those rays.   All the other hyperbolic components are contained in these rectangular-like regions.  The argument  and computation that shows the number of such rectangular-like regions whose diameter is bigger than any given $\epsilon$ is finite is similar to the argument and computation in  lemma~\ref{key2}.   Since there are infinitely such components, their diameters go to zero.   In particular,   if  $c_n$, $ n \to \infty$, is a sequence of centers of  distinct capture components   in one of these rectangular-like regions, any accumulation point   $c_n$ cannot be interior to a hyperbolic component  and so must be in $\calb$. 
\end{proof}
Note that it follows immediately that the Misiurewicz points on the boundary of $\calc_0$ are limits of sequences of virtual cycle parameters and centers whose orders tend to infinity.  This is illustrated in figure~\ref{central blowup}.   

In fact, the pictures of the parameter space show that capture components and shell components exhibit an interesting combinatorial structure.  Moreover, they indicate that the only parabolic parameters on $\partial \calc_0$ are cusps.  We plan to explore this further in future work.

\vspace*{20pt}
\noindent Tao Chen, Department of Mathematics, Engineering and Computer Science,
Laguardia Community College, CUNY,
31-10 Thomson Ave. Long Island City, NY 11101.\\
\noindent Email: tchen@lagcc.cuny.edu

\vspace*{5pt}
\noindent Linda Keen, Department of Mathematics, and  CUNY Graduate
School, New York, NY 10016, and Department of Mathematics, Lehman College of CUNY,
Bronx, NY 10468 \\
\noindent Email: LINDA.KEEN@lehman.cuny.edu, linda.keenbrezin@gmail.com

\end{document}